\documentclass[11pt]{article}   
   
\usepackage{latexsym,
amsfonts,
amsmath,
epsfig,
tabularx,
amsthm,
amssymb,
enumitem,
bm,
appendix,
float}

\usepackage{microtype}
\usepackage{color,tikz, pgfplots, pgfplotstable}
\definecolor{ManuGreen}{rgb}{0.2,0.6,0.2}
\definecolor{ManuGenta}{rgb}{0.5,0,1}
\definecolor{ManuOrange}{rgb}{1,0.6,0.3}
\definecolor{ManuRed}{rgb}{1,0,0.5}
\definecolor{NiceCrimson}{rgb}{0.6471, 0.1098, 0.1882}

\newcommand{\coarse}{\text{\sc Coarse}}
\newcommand{\apply}{\text{\sc Apply}}
\newcommand{\solve}{\text{\sc Solve}}
\newcommand{\rhs}{\text{\sc Rhs}}   
   
\usepackage[draft=false,colorlinks,bookmarksnumbered,linkcolor=black, citecolor=black]{hyperref}

\newcommand{\aref}[1]{\hyperref[#1]{Appendix}}

\usepackage{aliascnt}
\newtheorem{theorem}{Theorem}[section]

\newaliascnt{lem}{theorem}
\newtheorem{lemma}[lem]{Lemma}

\aliascntresetthe{lem}

\newaliascnt{prop}{theorem}
\newtheorem{prop}[prop]{Proposition}

\aliascntresetthe{prop}

\newaliascnt{cor}{theorem}
\newtheorem{corollary}[cor]{Corollary}

\aliascntresetthe{cor}

\newaliascnt{rem}{theorem}

\aliascntresetthe{rem}

\newaliascnt{def}{theorem}
\newtheorem{defi}[def]{Definition}

\aliascntresetthe{def}

\theoremstyle{definition}
\newaliascnt{ex}{theorem}

\aliascntresetthe{ex}     

\newaliascnt{ass}{theorem}
\newtheorem{ass}[ass]{Assumption}

\aliascntresetthe{ass}    
      
\renewcommand{\d}{\,\mathrm{d}}											
\renewcommand*{\epsilon}{\varepsilon}                                   
\renewcommand*{\rho}{\varrho}                                   		
\newcommand*{\nach}{\rightarrow}                                        
\newcommand*{\sep}{\; \vrule \;}                                        
\newcommand*{\N}{\mathbb{N}}                                            
\newcommand*{\R}{\mathbb{R}}                                            
\newcommand*{\C}{\mathbb{C}}                                            
\newcommand*{\G}{\mathcal{G}}                                           
\newcommand*{\F}{\mathcal{F}}                                         	
\newcommand*{\Co}{\mathcal{C}}                                         	
\newcommand*{\I}{\mathcal{I}}                                         	
\renewcommand*{\P}{\mathcal{P}}                                         
\newcommand*{\B}{\mathfrak{B}}                                          
\newcommand*{\leer}{\emptyset}                                          

\newcommand*{\abs}[1]{\left| #1 \right|}                                
\newcommand*{\norm}[1]{\left\| #1 \right\|}                             
\newcommand*{\normmod}[1]{\left\|\hspace{-1pt}\left| #1 \right|\hspace{-1pt}\right\|}                             
\newcommand*{\link}[1]{(\ref{#1})}                                      
\newcommand*{\distr}[2]{\left\langle #1, #2 \right\rangle}              
\newcommand*{\dist}[2]{\mathrm{dist}\!\left( #1, #2 \right)}            

\renewcommand{\tilde}[1]{ \widetilde{#1} }        						
\DeclareMathOperator{\supp}{supp}										
\DeclareMathOperator{\Id}{Id}											

\usepackage{footmisc}

\usepackage{tikz}
\usetikzlibrary{arrows,decorations.pathmorphing,backgrounds,positioning,fit,petri}

\usepackage[margin=1cm,format=hang,hangindent=0pt]{caption}

\title{Adaptive Wavelet BEM for Boundary Integral Equations:\\ 
Theory and Numerical Experiments\footnote{This work has been 
supported by the Deutsche Forschungsgemeinschaft DFG (DA 360/19-1)
and by the Swiss National Science Foundation SNSF (200021E-142224/1).}}
 
\author{S.~Dahlke\footnote{Corresponding author.}, \and 
H.~Harbrecht, \and M.~Utzinger, \and M.~Weimar}

\date{
\today} 

\begin{document}   
\maketitle   

\begin{abstract}
\noindent 
In this paper, we are concerned with the numerical treatment of boundary 
integral equations by means of the adaptive wavelet boundary element 
method (BEM). In particular, we consider the second kind Fredholm 
integral equation for the double layer potential operator on patchwise 
smooth manifolds contained in $\R^3$. The corresponding operator 
equations are treated by means of adaptive implementations that are 
in complete accordance with the underlying theory. The numerical 
experiments demonstrate that adaptive methods really pay off in 
this setting. The observed convergence rates fit together very well 
with the theoretical predictions that can be made on the basis of a 
systematic investigation of the Besov regularity of the exact solution.

\smallskip
\noindent \textbf{Keywords:} \textit{Besov spaces, weighted Sobolev spaces, 
adaptive wavelet BEM, non-linear approximation, integral equations, 
double layer potential operator, regularity, manifolds.}

\smallskip
\noindent \textbf{AMS Subject Classification:} 30H25,  35B65, 42C40, 45E99, 46E35, 47B38, 65T60.
\end{abstract}

\section{Introduction}
This paper is concerned with the theoretical analysis and 
numerical treatment of integral equations of the form
\begin{equation}\label{eq:generaleq}
S(u)=g \quad \text{on} \quad \partial\Omega,
\end{equation}
where $\partial \Omega$ denotes a patchwise smooth boundary 
of some bounded domain $\Omega\subset\R^d$. A typical example 
is given by the second kind Fredholm integral equation
\begin{equation}\label{DoubleLayerProb}
	S_{\mathrm{DL}}(u) 
	:=	\left( \frac{1}{2} \, \Id - K \right)\!(u)
	= g
	\quad \text{on} \quad \partial\Omega,
\end{equation}
where 
\begin{equation}\label{DLP}
	v \mapsto K (v) 
	:= \frac{1}{4\pi} \int_{\partial\Omega} v(y)\,
		\frac{\partial}{\partial \eta(y)} \frac{1}{\abs{\cdot-y\,}_2} \d\sigma(y)
\end{equation}
denotes the \emph{harmonic double layer potential operator} 
on $\partial{\Omega}$ which naturally arises from the so-called 
\emph{indirect method} for Dirichlet problems for Laplace's 
equation in $\Omega$. For details and further references,
see, e.g., \cite[Chapter 3.4]{SS11}, as well as 
\cite{DK87,JK81,JK95,K94,V84}.

Indirect methods obviously provide a reduction of the problem 
dimension. Therefore, in recent years, much effort has been 
spend to design efficient numerical schemes for the solution 
to these kind of equations such as, e.g., the multipole method
\cite{Rokh}, the panel clustering \cite{HackNov} or the adaptive
cross approximation \cite{Beb1}. In particular, one of the major 
computational bottlenecks is given by the fact that the 
discretization of \link{eq:generaleq} usually leads to 
densely populated matrices. In this regard, wavelet-based 
approaches provide striking advantages, since the wavelet's vanishing 
moment property can be used to design very powerful compression 
strategies \cite{DHS1,Sch06}. Although these algorithms turned out to 
be quite successful when dealing with classical test problems, 
for real-world problems involving millions of unknowns, it is 
necessary to further enhance efficiency by means of adaptive 
strategies. In the meantime, fully adaptive wavelet methods
which are guaranteed to converge with optimal order have 
been established \cite{DHS07,GantStev}. Although reliable 
error estimators for boundary integral operators exist and optimal 
convergence of traditional boundary element discretizations
have been proven, see, e.g., \cite{Faer,FeiKarMelPrae,Gant13}, 
we are not aware of any other method which is also 
{\em computationally\/} optimal.

Although the numerical experiments performed so far have been 
quite promising, one principle problem remains. Besides numerical 
evidence, it is of course very desirable to derive rigorous statements 
under which conditions the use of adaptive algorithms for \link{eq:generaleq} 
really pays off in practice. Given a dictionary $\B$ in the corresponding 
solution space, the best one could expect is that the adaptive algorithm 
realizes the convergence rate of the associated best $n$-term approximation 
with respect to $\B$ which is defined as follows. 

\begin{defi}
Let $\G$ denote a (quasi-) normed space and let $\B=\{b_1,b_2,\ldots\}$ 
be some countable subset of $\G$. Then
\begin{equation}\label{def_sigma}
	\sigma_n(u; \B, \G)
	:= \inf_{i_1,\ldots,i_n \in \N} \inf_{c_1,\ldots,c_n\in\C} 
		\norm{u - \sum_{m=1}^n c_m \, b_{i_m} \sep \G }, \qquad n\in\N,
\end{equation}
defines the error of the \emph{best $n$--term approximation} 
to some element $u$ with respect to the \emph{dictionary}
$\B$ in the (quasi-) norm of $\G$.
\end{defi}

Since $i_1,\ldots,i_n$ and $c_1,\ldots,c_n$ may depend on 
$u$ in an arbitrary way, this reflects how well we can 
approximate $u$ using a finite linear combination of elements 
in $\B$. These linear combinations obviously form a highly 
non-linear manifold in the linear space $\G$.  
Of course, to get 
a reasonable result concerning the achievable rate of best 
$n$--term approximations, additional information (such as 
further smoothness properties) of the target function $u$ 
is needed. This is usually modeled by its membership in 
some additional (quasi-) normed space~$\F$. Hence, if 
$\F$ denotes such a space which is embedded into $\G$, 
then we may study the asymptotic decay of
\begin{equation*}
	\sigma_n(\F;\B,\G) := \sup_{\substack{u\in \F,\\\norm{u \sep \F} \leq 1}} 
		\sigma_n(u;\B, \G), \qquad \text{as}\quad n\to\infty.
\end{equation*}
If the dictionary consists of a wavelet basis, it is indeed possible 
to construct optimal adaptive algorithms. That is, these schemes 
are guaranteed to converge with optimal order (i.e., they realize 
the convergence rate of best $n$--term wavelet approximation 
as defined above), while their computational costs stay proportional 
to the used number of degrees of freedom \cite{CDD01, CDD02}. 
Therefore, we can state that, in the wavelet setting, adaptivity 
really pays if the convergence order of best $n$--term approximation 
is strictly higher than the corresponding rate for classical non-adaptive 
algorithms. In this regard, it has been shown that the convergence 
order of \link{def_sigma} w.r.t.\ the $L_2$--norm is determined by 
the maximal regularity $\alpha$ of the $d$-variate function under 
consideration in the so--called {\em adaptivity scale} of 
{\em Besov spaces} 
\begin{equation} \label{adaptivityscale}
B^{\alpha_\tau}_{\tau}(L_{\tau}) \qquad \text{with} 
	\qquad \frac{1}{\tau}=\frac{\alpha_\tau}{d} + \frac{1}{2}, 
	\quad 0<\alpha_\tau < \alpha,
\end{equation}
where for all $\alpha_\tau>0$,
\begin{equation}\label{rateBesov}
\sigma_n(B^{\alpha_\tau}_{\tau}(L_{\tau}); \B, L_2)
	\sim n^{-\alpha_\tau/d}, \qquad \text{as} \quad n\to\infty,
\end{equation}
see, e.g., \cite{DNS06,D98, DJP92}. 
On the other hand, the convergence order of classical (uniform) 
algorithms is given by $n^{-s/d}$, where $s$ denotes the maximal 
regularity of the exact solution in the $L_2$--Sobolev scale, i.e., 
$s=\sup\{\mu > 0 \sep u\in H^{\mu}\}$; see, e.g., 
\cite{DDD97, DNS06,Hac1992} for details. In conclusion, the 
use of adaptivity is justified if the smoothness $\alpha$ of the 
exact solution $u$ to \link{eq:generaleq} in the adaptivity scale 
\link{adaptivityscale} of Besov spaces is higher than its 
Sobolev regularity $s$. 

For partial differential equations on bounded domains, a lot of 
positive results in this direction exist; see, e.g., \cite{DD97}. 
Quite recently, in \cite{DahWei2015}, also a positive result for 
integral operators on two--dimensional patchwise smooth manifolds 
has been derived. It has turned out that for a large class of operators, 
including the second kind Fredholm integral equation \link{DoubleLayerProb}
for the double layer potential operator \link{DLP}, the Besov smoothness 
of the solution in the adaptivity scale can be up to twice as high as the Sobolev 
regularity, so that adaptivity definitely makes sense. 

Nevertheless, one important issue still has to be discussed.  Of course, 
the estimate in \link{rateBesov} is of asymptotic nature, so that it is not 
clear that the advantage of adaptivity can really be observed in numerical 
practice. For example, it might happen that some of the involved constants 
are so large that the asymptotic behavior becomes significant only for 
values of $n$ which are far beyond any practical feasibility. It is one of 
the major goals of this paper to convince the reader that this is actually 
not the case. In order to do so, we performed numerical experiments 
that go far beyond trivial toy problems. We considered  the double 
layer potential operator on patchwise smooth manifolds. Our test 
cases, e.g., the Fichera vertex, are chosen in such a way that 
non-trivial singularities in the solution may show up at the interfaces. 
It turns out that the adaptive wavelet BEM indeed completely resolves 
all the singularities without any a priori information on the refinement 
strategy. We designed test cases where the solutions provably possess 
a relatively small Sobolev smoothness, but twice as much Besov regularity.  
And indeed, the adaptive wavelet algorithm converges twice as fast as 
the corresponding uniform scheme that simply uses all wavelets up to 
a given refinement level. Finally, we like to emphasize that our numerical 
realizations are in complete accordance with the theory development in 
\cite{DHS07,GantStev}. 

The paper is organized as follows: In \autoref{sect:prelim}, we start with 
some preparations concerning the parametrization of surfaces. In addition, 
here we define a scale of weighted Sobolev spaces~$X^k_\rho(\partial\Omega)$ 
which is needed to formulate our theoretical results. We also present a short 
introduction to the theory of layer potentials as far as it is needed for our 
purposes (see \autoref{sect:layer}). Moreover, we discuss the basic 
properties of the wavelet bases that are required for the use in adaptive 
algorithms. In \autoref{sect:Besov}, we define Besov-type spaces 
$B^{\alpha}_{\Psi,q}(L_p(\partial\Omega))$ as introduced in \cite{DahWei2015}
and clarify their relations to best $n$--term wavelet approximation. Our 
main Besov regularity results for solutions to integral equations are briefly 
summarized in \autoref{sect:mainresults}. \autoref{sec:AWEM} is dedicated
to the adaptive wavelet method. We survey on the basic ingredients needed
to realize an algorithm that realizes asymptotically optimal complexity. By 
asymptotically optimal we mean that 
 any target accuracy can be achieved at a computational expense that stays proportional the number of degrees of freedom
that is needed to approximate  the solution  by  $n$--term approximation with the same accuracy. 
Numerical results are then presented in \autoref{sec:numerix}. 
They are in good agreement with the theory. Finally,
in \autoref{sec:conclusion}, we state some concluding remarks.

\textbf{Notation:} For families $\{a_{i}\}_{i\in\I}$ and $\{b_{i}\}_{i\in\I}$ of 
non-negative real numbers over a common index set we write $a_{i} 
\lesssim b_{i}$ if there exists a constant $c>0$ such that
\begin{equation*}
	a_{i} \leq c\cdot b_{i}
\end{equation*}
holds uniformly in $i\in\I$.
Consequently, $a_{i} \sim b_{i}$ means $a_{i} \lesssim b_{i}$ 
and $b_{i} \lesssim a_{i}$.

\section{Preliminaries}\label{sect:prelim}
\subsection[Surfaces and weighted Sobolev spaces]{Surfaces $\partial\Omega$ and weighted Sobolev spaces $X^k_{\rho}(\partial\Omega)$}\label{sect:boundaries}
In this paper, we consider Lipschitz surfaces $\partial\Omega$ 
which are boundaries of bounded, simply connected, closed 
domains $\Omega\subset\R^3$ with polyhedral structure and 
finitely many quadrilateral sides. W.l.o.g.\ we can assume all 
these sides to be flat with corresponding straight edges; cf.\ 
\cite[Remark 2.1]{DahWei2015}. 

We will pursue essentially two different (but equivalent) 
approaches to describe $\partial\Omega$, where both of 
them will be used later on. For the first approach, consider 
the patchwise decomposition 
\begin{equation}\label{partition}
	\partial\Omega = \bigcup_{i=1}^I \overline{F_i},
\end{equation}
where $\overline{F_i}$ denotes the closure of the $i$th 
(open) \emph{patch} of $\partial\Omega$ which is a subset 
of some affine hyperplane in $\R^3$, bounded by a closed 
polygonal chain connecting exactly four points (\emph{vertices} $\nu$
of~$\Omega$). Here, we only require that the partition 
\link{partition} is essentially disjoint in the sense that the 
intersection of any two patches $\overline{F_i} \cap 
\overline{F_\ell}$, $i\neq \ell$, is either empty, a common 
edge, or a common vertex of $\Omega$. Furthermore, we 
will assume the existence of (sufficiently smooth) diffeomorphic 
parametrizations
\begin{equation*}
	\kappa_i \colon [0,1]^2 \nach \overline{F_i}, \qquad i=1,\ldots,I,
\end{equation*}
which map the unit square onto these patches. Finally, 
we define the class of \emph{patchwise smooth functions} 
on $\partial\Omega$ by
\begin{equation*}
	C_{\mathrm{pw}}^\infty(\partial\Omega) 
	:= \left\{ u \colon \partial\Omega \nach \C \sep u 
	\text{ is globally continuous and } u\big|_{\overline{F_i}} 
	\in C^{\infty}\!\left( \overline{F_i}\right) \text{ for all } i \right\}.
\end{equation*}

In the second approach, the surface of $\Omega$ is modeled 
(locally) in terms of the boundary of its tangent cones $\Co_n$, 
subordinate to the vertices $\nu_1,\ldots,\nu_N$ of $\Omega$.
For $n\in\{1,\ldots,N\}$, the boundary of the infinite cone $\Co_n$ 
consists of $T_n\geq 3$ essentially disjoint, open plane sectors 
(called \emph{faces}) which will be denoted by $\Gamma^{n,1}, 
\ldots, \Gamma^{n,T_n}$, i.e.,
\begin{equation*}
	\partial\Co_n = \bigcup_{t=1}^{T_n} \overline{\Gamma^{n,t}}, \qquad n=1,\ldots,N.
\end{equation*}
It will be convenient to use local polar coordinates $(r,\phi)$ 
in each of these faces~$\Gamma^{n,t}$. Then, every function 
$f_n \colon \partial\Co_n \nach \C$ can be described by a finite 
collection of functions $(f_{n,1},\ldots,f_{n,T_n})$ of the variable 
$y := (r\,\cos(\phi), r\,\sin(\phi)) \in \R^2$. Moreover, if $\gamma_{n,t}$ 
denotes the opening angle of $\Gamma^{n,t}$, the quantities $r$ and
\begin{equation*}
	q(\phi):=\min\!\left\{ \phi, \gamma_{n,t} - \phi \right\} \in (0,\pi)
\end{equation*}
serve as a distance measure of the point $y$ to the face boundary; 
see \cite[Formula (7)]{DahWei2015} for details. Therefore, 
following~\cite{E92}, weighted Sobolev spaces on the boundary of 
the cone $\Co_n$ can be defined as the closure of all continuous, 
facewise smooth, compactly supported functions on~$\partial\Co_n$,
\begin{equation*}
	C_{0,\mathrm{fw}}^\infty(\partial\Co_n) 
	:= \left\{ f_n \in C_0(\partial\Co_n) \sep f_{n,t} \in C^{\infty}\!
	\left( \overline{\Gamma^{n,t}}\right) \text{ for all } t=1,\ldots,T_n \right\},
\end{equation*}
with respect to the norm $\norm{f_n \sep X^k_{\rho}(\partial\Co_n)}$ 
given by
\begin{equation}\label{eq:normCn}
	\norm{f_n \sep L_2(\partial\Co_n)} 
		+ \sum_{t=1}^{T_n} \sum_{\substack{\beta
	=(\beta_r,\beta_\phi)\in\N_0^2\\1\leq\abs{\beta}\leq k}} 
	\norm{ \left(1+\frac{1}{r}\right)^{\rho} (q\,r)^{\beta_r} 
	\left( \frac{\partial}{\partial r} \right)^{\beta_r} q^{\beta_\phi-\rho} 
	\left( \frac{\partial}{\partial \phi} \right)^{\beta_\phi} f_{n,t} 
	\sep L_2\!\left(\Gamma^{n,t}\right) }.
\end{equation}
That is, we let
\begin{equation*}
		X^k_{\rho}(\partial\Co_n)
		:= \overline{C_{0,\mathrm{fw}}^\infty (\partial\Co_n)}^{\, 
			\norm{\cdot \sep X^k_{\rho}(\partial\Co_n)} },
\end{equation*}
where, as usual, the sum over an empty set is to be interpreted 
as zero, $k\in \N$ is the smoothness of the space, and $\rho\in[0,k]$ 
controls the strength of the weight. 

In order to analyze functions $u$ defined on the whole surface  
$\partial\Omega \subset\bigcup_{n=1}^N \partial\Co_n$, we localize 
them with the help of a special resolution of unity $(\varphi_n)_{n=1}^N$ 
to cone faces near the vertices $\nu_1$, \ldots, $\nu_N$ of $\Omega$. 
Hence, for $u\colon \partial \Omega \nach \C$, $k\in\N$, and 
$0\leq\rho\leq k$, we let
\begin{equation}\label{def_X}
	\norm{u \sep X^k_\rho(\partial\Omega)} 
	:= \sum_{n=1}^N \norm{\varphi_n \,u \sep X^k_\rho(\partial\Co_n)}
\end{equation}
and define the weighted Sobolev space on $\partial\Omega$ as
\begin{equation*}
	X^k_{\rho}(\partial\Omega)
		:= \overline{C_{\mathrm{pw}}^\infty (\partial\Omega)}^{\, 
		\norm{\cdot \sep X^k_{\rho}(\partial\Omega)} }.
\end{equation*}
For details, the interested reader is again referred to \cite{DahWei2015}.

\subsection{Layer potentials}\label{sect:layer}
We shall be concerned with the solution of the 
Dirichlet problem for the Laplacian
\begin{equation*}
  \Delta U = 0\quad\text{in}\ \Omega, \qquad
	U = g\quad\text{on}\ \partial\Omega,
\end{equation*}
on a domain $\Omega\subset\R^3$ with patchwise smooth 
boundary $\partial\Omega$ by means of \emph{harmonic 
double layer potentials}. To that end, let $\sigma$ denote 
the canonical surface measure. Since this surface is assumed 
to be Lipschitz, for $\sigma$-a.e.\ $x\in\partial\Omega$ there 
exists the outward pointing normal vector $\eta(x)$. By 
$\partial/\partial\eta(x)$ we denote the corresponding conormal 
derivative in $x\in\partial\Omega$. Making the potential ansatz
\begin{equation}\label{indirect}
  U(x) := \frac{1}{4\pi} \int_{\partial\Omega} u(y)\,
		\frac{\partial}{\partial \eta(y)} \frac{1}{\abs{x-y\,}_2} \d\sigma(y),
			\qquad x\in\Omega,
\end{equation}
and letting $x$ tend to $\partial\Omega$, we arrive, in view of the 
jump condition, at the second kind Fredholm integral equation 
\link{DoubleLayerProb} for the unknown \emph{density} $u$. The 
integral operator under consideration is of order zero,
\begin{equation*}
S_{\mathrm{DL}} = \left( \frac{1}{2} \, \Id - K \right):
	L_2(\partial\Omega)\to L_2(\partial\Omega),
\end{equation*}
where $\Id$ denotes the identical mapping on $\partial\Omega$ 
and $K$ is the harmonic double layer potential operator as 
defined in \link{DLP}.
One speaks here of the \emph{indirect method}: The sought 
solution $U\in H^1(\Omega)$ is not computed directly, but 
indirectly via the evaluation of the potential \link{indirect},
as soon as the potential's density $u$ is known. 

When it comes to the numerical approximation of the solution $u$ to
\link{DoubleLayerProb}, well-posedness of the problem is essential.
Invertibility within the (unweighted) Sobolev scale $H^s(\partial\Omega)$ 
is known as \emph{Verchota's Theorem} \cite{V84}; 
see also \cite[Remark A.5]{E92}.
\begin{prop}[Verchota~{\cite[Theorem 3.3(iii)]{V84}}]\label{prop:Verchota}
For all $s\in [0,1]$, the bounded linear operator
$S_\mathrm{DL} \colon H^{s}(\partial\Omega) \nach 
H^{s}(\partial\Omega)$ is invertible.
\end{prop}

However, for our purposes, bounded invertibility in $H^s(\partial\Omega)$
is not enough. As we shall see later in \autoref{sect:mainresults}, also
regularity estimates in the weighted Sobolev scale $X_\rho^k(\partial\Omega)$
as introduced in \autoref{sect:boundaries} are needed.
In this context, the subsequent result taken from \cite{E92} is
particularly useful.

\begin{prop}[Elschner~{\cite[Remark~4.3]{E92}}]\label{prop:Elschner}
There exists a constant $\rho_0 \in (1, 3/2)$, depending on 
the surface $\partial\Omega$, such that the following is true:
For all $0 \leq \rho < \rho_0$ and every $k\in\N$ with $\rho \leq k$,
the bounded linear operator $S_\mathrm{DL} \colon X_\rho^k(\partial\Omega) 
\nach X_\rho^k(\partial\Omega)$ is invertible.
\end{prop}

Note that in the notation of \cite{V84} $K$ is replaced by $-K$ and 
that the operator $K$ considered in \cite{E92} differs from our notation 
by a factor of $1/2$. Nevertheless, the whole analysis carries over.

\subsection{Wavelet bases}
During the past years, wavelets on domains $\Omega \subseteq \R^d$ 
have become a powerful tool in both, pure and applied mathematics. 
More recently, several authors proposed various constructions of 
wavelet systems extending the idea of multiscale analysis to manifolds 
based on patchwise descriptions such as \link{partition}; see, e.g., 
\cite{CTU99,CTU00,CM00,DS99,HaSch04,HS06}. Later on in this 
paper, bases of these kinds will be used to define new types of Besov 
spaces on $\partial \Omega$. Therefore, in this subsection, we collect 
some basic properties that will be needed for this purpose.

With the help of the parametric liftings $\kappa_i$, $i=1,\ldots,I$, 
an inner product for functions $u,v\colon\partial\Omega\nach\C$ 
can be defined patchwise by
\begin{equation*}
	\distr{u}{v} := \sum_{i=1}^I \distr{u\circ \kappa_i}{v\circ\kappa_i}_{\square},
\end{equation*}
where $\distr{\cdot}{\cdot}_{\square}$ denotes the usual $L_2$-inner 
product on the square $[0,1]^2$. Since all $\kappa_i$ are assumed to 
be sufficiently smooth the norm induced by $\distr{\cdot}{\cdot}$ can 
be shown to be equivalent to the norm in $L_2(\partial\Omega)$:
\begin{equation}\label{normeq}
	\normmod{\cdot}_0 := \sqrt{\distr{\cdot}{\cdot}} \sim \norm{\cdot \sep L_2(\partial\Omega)},
\end{equation}
see, e.g., Formula (4.5.3) in~\cite{DS99}.

Most of the known wavelet  constructions are based on tensor products 
of boundary-adapted wavelets/scaling functions (defined on intervals) 
which are finally lifted to the patches $F_i$ describing the surface 
$\partial\Omega$. A wavelet basis $\Psi=(\Psi^{\partial\Omega}, 
\tilde{\Psi}^{\partial\Omega})$ on $\partial\Omega$ then consists 
of two collections of functions $\psi^{\partial\Omega}_{j,\xi}$ and 
$\tilde{\psi}^{\partial\Omega}_{j,\xi}$, respectively, that form 
($\distr{\cdot}{\cdot}$-biorthogonal) Riesz bases for $L_2(\partial\Omega)$. 
In particular, every $u\in L_2(\partial\Omega)$ has a unique expansion
\begin{equation}\label{expansion}
	u = P_{j^\star-1}(u) + \sum_{j\geq j^\star} \sum_{\xi \in {\nabla}_j^{\partial\Omega}} 
		\distr{u}{\tilde{\psi}^{\partial\Omega}_{j,\xi}} \, \psi^{\partial\Omega}_{j,\xi}
\end{equation}
satisfying
\begin{equation*}
	\norm{u\sep L_2(\partial\Omega)} 
	\sim \norm{P_{j^\star-1}(u) \sep L_2(\partial\Omega)} 
	+ \left( \sum_{j\geq j^\star} \sum_{\xi \in {\nabla}_j^{\partial\Omega}} 
	\abs{\distr{u}{\tilde{\psi}^{\partial\Omega}_{j,\xi}}}^2 \right)^{1/2}.
\end{equation*}
Therein, $P_{j^\star-1}$ denotes the biorthogonal projector that maps 
$L_2(\partial\Omega)$ onto the finite dimensional span of all 
generators  on the  coarsest level  $j^\star-1$. 

In the sequel, we will require  that the wavelet basis under 
consideration satisfies all the conditions  collected in the 
following assumption.

\begin{ass}\label{ass:basis}
\noindent
\begin{itemize}
\item[(I)] As indicated in \link{expansion}, both (the primal and the dual) 
systems are indexed by their level of resolution $j\geq j^\star$, as well as 
their location (and type) $\xi \in \nabla^{\partial\Omega}$. We assume that 
this collection of grid points on the surface $\partial\Omega$ can be split 
up according to the levels $j$ and the patches $\overline{F_i}$:
\begin{equation*}
	{\nabla}^{\partial\Omega} 
	= \bigcup_{j=j^\star}^\infty {\nabla}_j^{\partial\Omega},
	\quad \text{where, for all } j\geq j^\star,
	\quad 
	{\nabla}_j^{\partial\Omega} = \bigcup_{i=1}^I {\nabla}_j^{F_i}
	\quad \text{with}\quad   \#\nabla_j^{F_i} \sim 2^{2j}.
\end{equation*}

\item[(II)]  All dual wavelets are $L_2$-normalized: 
\begin{equation}\label{normalized}
	\norm{\tilde{\psi}^{\partial\Omega}_{j,\xi} \sep L_2(\partial\Omega)} \sim 1
	\qquad \text{for all} \qquad j\geq j^\star, \ \xi\in {\nabla}_j^{\partial\Omega}.
\end{equation}

\item[(III)] We assume that all elements $\tilde{\psi}^{\partial\Omega}_{j,\xi} 
\in \tilde{\Psi}^{\partial\Omega}$ are compactly supported on $\partial\Omega$. 
Furthermore, we assume that their supports  contain the  corresponding 
grid point~$\xi$ and satisfy
\begin{equation}\label{supp}
	\abs{\supp \tilde{\psi}^{\partial\Omega}_{j,\xi}} \sim 2^{-2j}
	\qquad \text{for all} \qquad j\geq j^\star, \ \xi\in {\nabla}_j^{\partial\Omega}.
\end{equation}

\item[(IV)] Consider the set $\Pi_{\tilde{d}-1}([0,1]^2)$ of polynomials 
$\P$ on the unit square which have a total degree $\deg\P$ strictly 
less than $\tilde{d}$. Then, we assume that the  dual system 
$\tilde{\Psi}^{\partial\Omega}$ satisfies
\begin{equation}\label{vanish}
	\distr{\P}{\tilde{\psi}^{\partial\Omega}_{j,\xi}\circ \kappa_i}_\square = 0
	\qquad
	\text{for all}
	\qquad \P\in \Pi_{\tilde{d}-1}([0,1]^2),
\end{equation}
whenever $\tilde{\psi}^{\partial\Omega}_{j,\xi} \in \tilde{\Psi}^{\partial\Omega}$ 
is completely supported in the interior of some patch $F_i\subset\partial\Omega$, 
$i\in\{1,\ldots,I\}$. This property is commonly known as \emph{vanishing moment 
property} of order $\tilde{d}\in\N$.

\item[(V)] The number of dual wavelets at level $j$ with distance $2^{-j}$ 
to one of the patch boundaries is of order  $2^{j}$, i.e.,
\begin{equation}\label{number_wavelets1}
\#\!\left\{ \xi \in \nabla^{\partial \Omega}_j \sep 0 
	< \dist{\supp \tilde{\psi}^{\partial\Omega}_{j,\xi}}{\bigcup_{i=1}^I \partial F_i} 
	\lesssim 2^{-j} \right\} \sim 2^{j} 
\quad \text{for all} \quad j\geq j^\star.
\end{equation}
Moreover, for the dual wavelets intersecting one of the 
patch interfaces, we assume that
\begin{equation*}
\#\!\left\{ \xi \in \nabla^{\partial \Omega}_j \sep 
	\supp \tilde{\psi}^{\partial\Omega}_{j,\xi} \cap 
	\bigcup_{i=1}^I \partial F_i \neq \leer \right\} \lesssim 2^{j}
\quad \text{for all} \quad j\geq j^\star.
\end{equation*}

\item[(VI)] Every point $x\in\partial\Omega$ is contained in the 
supports of a uniformly bounded number of dual wavelets at level $j$:
\begin{equation}\label{finite_overlap}
\#\!\left\{ \xi \in \nabla^{\partial \Omega}_j \sep x\in\supp \tilde{\psi}^{\partial\Omega}_{j,\xi} \right\} \lesssim 1
\quad \text{for all} \quad j\geq j^\star \quad \text{and each} \quad x\in\partial\Omega.
\end{equation}

\item[(VII)] Finally, we assume that the Sobolev spaces 
$H^s(\partial\Omega)=W^s(L_2(\partial\Omega))$ in the 
scale
\begin{equation*}
	-\frac{1}{2} < s < \min\!\left\{\frac{3}{2}, s_{\partial\Omega}\right\}
\end{equation*}
can be characterized by the decay of wavelet expansion coefficients, that is 
\begin{equation}\label{equiv_sobolev_norm}
		\norm{u\sep H^s(\partial\Omega)} 
		\sim \norm{P_{j^\star-1}(u) \sep L_2(\partial\Omega)} 
		+ \left( \sum_{j\geq j^\star} \sum_{\xi \in {\nabla}_j^{\partial\Omega}} 2^{2sj} 
		\abs{\distr{u}{\tilde{\psi}^{\partial\Omega}_{j,\xi}}}^2 \right)^{1/2}.
\end{equation}

Here, the spaces for negative $s$ are defined by duality and 
$s_{\partial\Omega} \geq 1$ depends on the interior angles 
between different patches $F_i$ of the manifold under 
consideration; cf.\ \cite[Section~4.5]{DS99}.
\end{itemize}
\end{ass}

Fortunately, all these  assumptions are satisfied for all the 
constructions we mentioned at the beginning of this subsection. 
In particular, the \emph{composite wavelet basis} as constructed 
in~\cite{DS99} is a typical example which will serve as our main 
reference. Note that although those wavelets are usually at most 
continuous across patch interfaces, they are able to capture 
arbitrary high smoothness in the interior by increasing the order 
of the underlying boundary-adapted wavelets.

\section{Besov regularity}\label{sect:Besov}
\subsection[Besov-type function spaces]{Besov-type function spaces on $\partial\Omega$}
Besov spaces essentially generalize the concept of Sobolev 
spaces. On $\R^d$ they are typically defined using harmonic 
analysis, finite differences, moduli of smoothness, or interpolation 
techniques. Characteristics (embeddings, interpolation results, and 
approximation properties) of these scales of spaces are then obtained 
by reducing the assertion of interest to the level of sequences spaces 
by means of characterizations in terms of building blocks (atoms, local 
means, quarks, or wavelets). To mention at least a few references, the 
reader is referred to the monographs \cite{RS96,T06}, as well as to the 
articles \cite{DJP92,FJ90,KMM07}. This list is clearly not complete.

Besov spaces on manifolds such as boundaries of domains in $\R^d$ 
can be defined as trace spaces or via pullbacks based on (overlapping) 
resolutions of unity. In general, traces of wavelets are not wavelets anymore, 
and if we use pullbacks, then wavelet characterizations are naturally limited 
by the global smoothness of the underlying manifold. Therefore, let us recall 
a notion of Besov-type spaces from \cite[Definition~4.1]{DahWei2015} which is \emph{based on} expansions w.r.t.\ some 
biorthogonal wavelet Riesz basis $\Psi=(\Psi^{\partial\Omega}, 
\tilde{\Psi}^{\partial\Omega})$ satisfying the conditions of the 
previous section which we assume to be given fixed: 

\begin{defi}
A tuple of real parameters $(\alpha, p, q)$ is said to be 
\emph{admissible} if
\begin{equation}\label{parameter}
	\frac{1}{2} \leq \frac{1}{p} \leq \frac{\alpha}{2} + \frac{1}{2}
	\qquad \text{and} \qquad 
	0 < q 
	\leq \begin{cases}
		2 & \text{if\/}\ 1/p = \alpha/2 + 1/2, \\   
		\infty, & \text{otherwise}.
	\end{cases}
\end{equation}
Given a wavelet basis $\Psi=(\Psi^{\partial\Omega}, 
\tilde{\Psi}^{\partial\Omega})$ on $\partial\Omega$
and a tuple of admissible parameters $(\alpha, p, q)$
let $B_{\Psi,q}^\alpha(L_p(\partial\Omega))$ denote the 
collection of all complex-valued functions $u\in L_2(\partial\Omega)$ 
such that the (quasi-) norm
\begin{equation*}
		\norm{u \sep B_{\Psi,q}^\alpha(L_p(\partial\Omega))}
		:= \norm{P_{j^\star-1}(u) \sep L_p(\partial\Omega)} + \left( \sum_{j\geq j^\star} 2^{j \left(\alpha+2\left[ \frac{1}{2} - \frac{1}{p} \right] \right)q} \left[ \sum_{\xi \in {\nabla}_j^{\partial\Omega}} \abs{\distr{u}{\tilde{\psi}^{\partial\Omega}_{j,\xi}}}^p \right]^{q/p} \right)^{1/q}
\end{equation*}
is finite (with the usual modification if $q=\infty$).
\end{defi}

In the remainder of this subsection, we collect some basic 
properties of the Besov-type spaces introduced above. To start with, 
we note that all spaces $B_{\Psi,q}^\alpha(L_p(\partial\Omega))$ 
are quasi-Banach spaces, Banach spaces if and only if $\min\{p,q\} \geq 1$, and Hilbert spaces if and only if $p=q=2$.

Formally, different bases $\Psi$ might lead to different function 
spaces even if all remaining parameters $(\alpha,p,q)$ that 
determine the spaces may coincide. Nevertheless, in \cite{Wei2016},
it has been shown that under very natural conditions the resulting 
Besov spaces coincide up to equivalent norms. These conditions 
are fortunately satisfied by most of the wavelet bases which are 
available in the literature.

Finally, let us recall an assertion which clarifies the relation of the 
scales $B_{\Psi,q}^{\alpha}(L_{p}(\partial\Omega))$ and best 
$n$--term wavelet approximation; cf.\ \cite[Proposition~4.7]{DahWei2015}.
For a visualization of the involved embeddings we refer to the left \emph{DeVore-Triebel diagramm} in \autoref{fig:DeVoreTriebel2}.

\begin{prop}\label{prop:nterm}
	For $\gamma \in \R$, let $(\alpha+\gamma,p_0,q_0)$ and 
	$(\alpha,p_1,q_1)$ be admissible parameter tuples. If 
	$\gamma > 2 \cdot \max\!\left\{0,\frac{1}{p_0} - \frac{1}{p_1}\right\}$,
	then
	\begin{equation*}
		\sigma_n \!\left( B_{\Psi,q_0}^{\alpha+\gamma}(L_{p_0}(\partial\Omega));
			\Psi^{\partial\Omega},B_{\Psi,q_1}^\alpha(L_{p_1}(\partial\Omega)) \right) 
		\sim n^{-\gamma/2}.
	\end{equation*}
	Moreover, if $\gamma = 2 \cdot \max\!\left\{0,\frac{1}{p_0} - \frac{1}{p_1}\right\}$ 
	and $q_0 \leq q_1$, then
	\begin{equation*}
		\sigma_n \!\left( B_{\Psi,q_0}^{\alpha+\gamma}(L_{p_0}(\partial\Omega));
		\Psi^{\partial\Omega},B_{\Psi,q_1}^\alpha(L_{p_1}(\partial\Omega)) \right) 
		\sim n^{-\min\{\gamma/2,\, 1/q_0-1/q_1\}}. 
	\end{equation*}
\end{prop}

In view of our application to integral equations, we are particularly 
interested in the rate of convergence of best $n$--term wavelet 
approximation to solutions $u\in B^{s'+\gamma}_{\Psi,\tau}
(L_\tau(\partial\Omega))$ with
\begin{equation}\label{eq:parameters}
s'\in[0,\max\{3/2,s_{\partial\Omega}\}), \quad \gamma\geq 0, 
	\quad \text{and} \quad \tau:=(\gamma/2+1/2)^{-1}
\end{equation}
w.r.t.\ the norm in $H^{s'}(\partial\Omega)$.

\begin{corollary}
For $s',\gamma$, and $\tau$ given by \link{eq:parameters}, we have
	\begin{equation*}
		\sigma_n \!\left( B_{\Psi,\tau}^{s'+\gamma}(L_{\tau}(\partial\Omega));
		\Psi^{\partial\Omega},H^{s'}(\partial\Omega) \right) 
		\sim n^{-\gamma/2}\quad \text{as}\quad n\to\infty.
	\end{equation*}
\end{corollary}
\begin{proof}
Apply \autoref{prop:nterm} with $p_0:=q_0:=\tau$, $p_1:=q_1:=2$, and $\alpha:=s'$. 
\end{proof}

\subsection{Main regularity results}\label{sect:mainresults}
Throughout the whole section, $\partial\Omega$ denotes 
the patchwise smooth boundary of some three-dimensional 
domain $\Omega$, as described in \autoref{sect:boundaries}.
Moreover, we assume to be given a biorthogonal wavelet Riesz 
basis $\Psi=(\Psi^{\partial\Omega},\tilde{\Psi}^{\partial\Omega})$ 
on $\partial\Omega$, satisfying the requirements stated 
in \autoref{ass:basis}.

Given an operator $S$ and a right-hand side $g\colon 
\partial\Omega\nach \C$, we like to solve the equation
\begin{equation}\label{OperatorS}
	S(u)=g \quad \text{on} \quad \partial\Omega
\end{equation}
for $u\colon\partial\Omega\nach\C$. In particular, we 
are interested in the asymptotic behavior of the error of 
best $n$--term wavelet approximation to $u$, measured 
in the norm of $H^{s'}(\partial\Omega)$ for some $s' \geq 0$.

\begin{theorem}[{\cite[Theorem~5.6]{DahWei2015}}]\label{thm:general_eq}
Assume $\tilde{d}\in\N$, $k\in\{1,2,\ldots,\tilde{d}\}$, as well 
as $\rho\in(0,k)$, and let $(s,p,p)$ be an admissible tuple of 
parameters with $s>0$. Whenever the solution $u$ to \link{OperatorS} 
is contained in the intersection of $B_{\Psi,p}^s(L_p(\partial\Omega))$ 
and $X^k_\rho(\partial\Omega)$, then it also belongs to the Besov-type 
space $B_{\Psi,\tau}^\alpha(L_\tau(\partial\Omega))$ for all tuples 
$(\alpha,\tau,\tau)$ with
\begin{equation*}
	\frac{1}{\tau} = \frac{\alpha}{2}+\frac{1}{2}
	\quad\, \text{and} \quad\,
	0 \leq \alpha < 2 \alpha^\star, 
	\quad\, \text{where} \quad\,
	\alpha^\star = \min\!\left\{\rho, k-\rho, s-\left(\frac{1}{p}-\frac{1}{2} \right)\right\}.
\end{equation*}	
Moreover, for every $0 \leq s' < \min\!\left\{ 3/2, s_{\partial\Omega} \right\}$,
satisfying
\begin{equation}\label{bound_s}
	s-s' \geq 2 \left( \frac{1}{p} - \frac{1}{2}  \right),
\end{equation}
we have $\sigma_n\!\left( u; \Psi^{\partial\Omega}, 
H^{s'}(\partial\Omega) \right) \lesssim n^{-\gamma/2}$
as $n\nach\infty$ for all $\gamma < \gamma^\star$, where
	\begin{equation*}
		\gamma^\star := s-s' + \Theta \cdot (2\alpha^\star - s)\geq 0
		\qquad \text{and} \qquad
		\Theta := 1- \frac{s'}{s-2 \left( 1/p - 1/2 \right)} \in [0,1].
	\end{equation*}
\end{theorem}

\begin{figure}[ht!]
\begin{center}
\begin{tikzpicture}
\draw [->,>=stealth'] (0,-0.05) -- (0,6) node[above] {$\alpha$};
\draw [->,>=stealth'] (-0.05,0) -- (5,0) node[right] {$\frac{1}{p}$};

\draw [gray, thick] (1,3) -- (4,6);
\shade [left color=gray!40, right color=gray!5, shading angle=180] (1,3)--(1,6)--(4,6)--(1,3);

\draw [thick] (1,0) -- (5,4);
\draw [very thick] (1,0) -- (1,6);
\fill (1,0) circle (2pt) node[anchor=south east] {$L_2$};

\draw (-0.3,3) node {$s'$};
\draw [dashed] (-0.05,3) -- (1,3);
\draw (1,0) -- (1,-0.05);
\draw (1,-0.5) node {$\frac{1}{2}$};
\fill (1,3) circle (2pt) node[anchor=north east] {$H^{s'}$};

\draw (-0.05,4.5) node[anchor=east] {$s_{\partial\Omega}$} -- (0.05,4.5);

\draw (-0.05,3.5) node[anchor=east] {$\alpha+\gamma$};
\draw [dashed] (-0.05,3.5) -- (3,3.5);
\fill (3,3.5) circle (2pt) node[anchor=south] {$B^{\alpha+\gamma}_{\Psi,p_0}(L_{p_0})$};
\draw [dashed] (3,-0.05) -- (3,0.7);
\draw [dashed] (3,1.4) -- (3,3.5);
\draw (3,-0.5) node {$\frac{1}{p_0}$};

\draw [->,>=stealth'] (3,3.5) -- (4,3.5);
\draw [dotted] (4,3.5)--(4.5,3.5);
\draw [->,>=stealth'] (3,3.5) -- (2,2.5);
\draw [dotted] (2,2.5)--(1,1.5);

\draw (-0.3,1) node {$\alpha_\tau$};
\draw [dashed] (-0.05,1.1) -- (2.1,1.1);
\fill (2.1,1.1) circle (2pt) node[anchor=west] {$B^{\alpha_\tau}_{\Psi,\tau}(L_\tau)$};
\draw [dashed] (2.1,-0.05) -- (2.1,1.1);
\draw (2.1,-0.5) node {$\frac{1}{\tau}$};
\end{tikzpicture}
		\hfill
\begin{tikzpicture}
\draw [->,>=stealth'] (0,-0.05) -- (0,6) node[above] {$\alpha$};
\draw [->,>=stealth'] (-0.05,0) -- (7,0) node[right] {$\frac{1}{p}$};

\shade [left color=gray!40, right color=gray!5, shading angle=180] (1,1)--(1,6.5)--(6.5,6.5)--(1,1);

\draw [thin] (1,0) -- (6.5,5.5);
\draw [thin] (1,0) -- (1,6);
\draw [thin] (1,1) -- (5.5,5.5);
\fill (1,0) circle (2pt) node[anchor=south east] {$L_2$};

\draw (-0.3,2.5) node {$s$};
\draw [dashed] (-0.05,2.5) -- (1,2.5);
\draw (1,0) -- (1,-0.05);
\draw (1,-0.5) node {$\frac{1}{2}$};
\fill (1,2.5) circle (2pt) node[anchor=south east] {$H^{s}$};

\draw (-0.3,1) node {$s'$};
\draw [dashed] (-0.05,1) -- (1,1);
\fill (1,1) circle (2pt) node[anchor=south east] {$H^{s'}$};

\draw (-0.3,5) node {$2s$};
\draw [dashed] (-0.05,5) -- (6,5);
\draw [dashed] (6,-0.05) -- (6,5);
\draw (6,-0.5) node {$\frac{1}{\tau}$};

\fill (6,5) circle (2pt) node[anchor=north west] {$B^\alpha_{\Psi,\tau}(L_\tau)$};
\draw [thick, dotted] (1,2.5) -- (4,4);
\draw [->,>=stealth', dotted] (4,4) -- (5.8,4.9);

\draw (-0.3,4) node {$\alpha_\theta$};
\draw [dashed] (-0.05,4) -- (4,4);
\fill (4,4) circle (2pt) node[anchor=south east] {$B^{\alpha_\theta}_{\Psi,p_\theta}(L_{p_\theta})$};
\draw [dashed] (4,-0.05) -- (4,4);
\draw (4,-0.5) node {$\frac{1}{p_\theta}$};
\end{tikzpicture}
	\end{center}
\caption{DeVore-Triebel diagrams visualizing the area of admissible 
parameters, as well as embeddings related to 
\autoref{prop:nterm} (left), and \autoref{thm:general_eq} (right).\label{fig:DeVoreTriebel2}}
\end{figure}
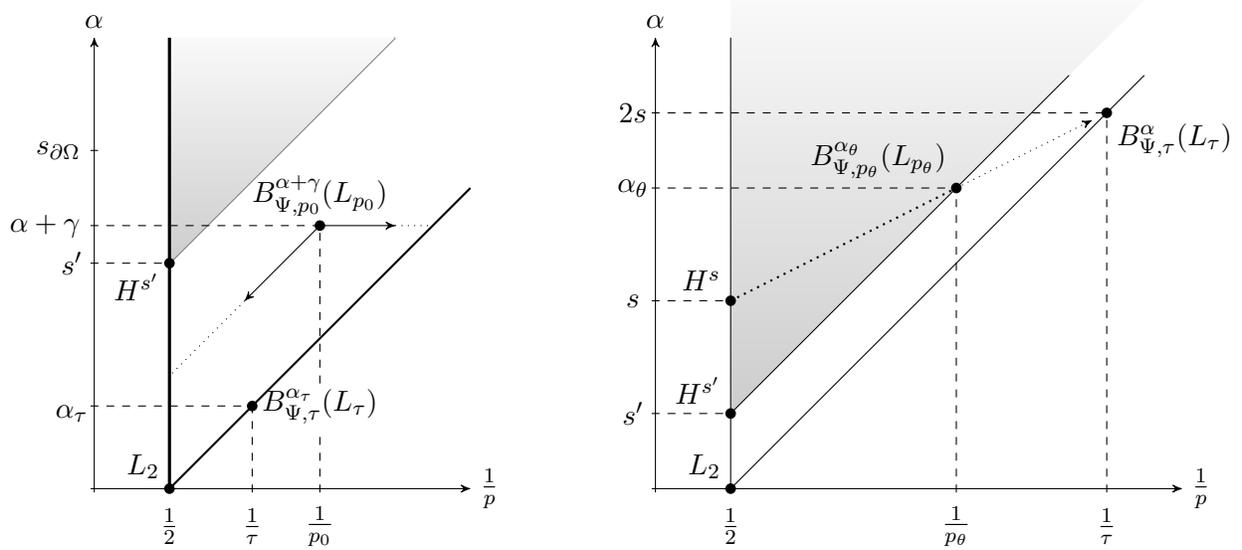

The second DeVore-Triebel diagram in \autoref{fig:DeVoreTriebel2} 
illustrates a special case of \autoref{thm:general_eq}. There, we have 
chosen $p=2$ and $0<s'<s<\min\{\rho,k-\rho,3/2,s_{\partial\Omega}\}$, 
such that particularly $\alpha^\star=s$ and
$B_{\Psi,p}^s(L_p(\partial\Omega))=H^s(\partial\Omega)$. The dotted line corresponds to the scale of spaces 
$B_{\Psi,p_\Theta}^{\alpha_{\Theta}}(L_{p_\Theta}(\partial\Omega))$ 
which can be reached by complex interpolation of $H^s(\partial\Omega)$ 
and $B_{\Psi,\tau}^\alpha(L_\tau(\partial\Omega))$. 
Interpolation is necessary since only those spaces which belong to the 
shaded area can be embedded into $H^{s'}(\partial\Omega)$. For this 
special choice of the parameters we obtain that $u\in H^s(\partial\Omega) 
\cap X^k_\rho(\partial\Omega)$ can be approximated in the norm of 
$H^{s'}(\partial\Omega)$ at a rate arbitrarily close to $\gamma^\star/2 = s-s'$, 
whereas the rate of convergence for best $n$--term wavelet approximation 
to an arbitrary function $u\in H^s(\partial\Omega)$ is $(s-s')/2$; see 
\autoref{prop:nterm}. Hence, incorporating the additional knowledge 
about weighted Sobolev regularity (membership in $X^k_\rho
(\partial\Omega)$) allows to improve the rate of convergence 
up to a factor of two. 

For the special case of the double layer operator $S:=S_{\mathrm{DL}}$, 
we obtain the following theorem. Its proof is based on a combination of Propositions \ref{prop:Verchota} and \ref{prop:Elschner} with \autoref{thm:general_eq}; see \cite{DahWei2015}.

\begin{theorem}[{\cite[Theorem~5.8]{DahWei2015}}]\label{thm:doublelayer}
Let $s\in(0,1)$, as well as $k\in\N$, and $\rho\in(0,\min\{\rho_0,k\})$ 
for some $\rho_0 \in (1, 3/2)$ depending on the surface $\partial\Omega$. 
Moreover let $\alpha$ and $\tau$ be given such that
\begin{equation*}
	\frac{1}{\tau} = \frac{\alpha}{2}+\frac{1}{2}
	 \quad\, \text{and} \quad\,
	 0 \leq \alpha < 2 \cdot \min\{\rho, k-\rho, s\}
\end{equation*}	
and let the Besov-type space $B_{\Psi,\tau}^\alpha(L_\tau(\partial\Omega))$ 
be constructed with the help of a wavelet basis $\Psi=(\Psi^{\partial\Omega},
\tilde{\Psi}^{\partial\Omega})$ possessing vanishing moments of order 
$\tilde{d} \geq k$. Then, for every right-hand side $g\in H^s(\partial\Omega)
\cap X_\rho^k(\partial\Omega)$, the double layer equation \link{DoubleLayerProb} 
has a unique solution $u\in B_{\Psi,\tau}^\alpha(L_\tau(\partial\Omega))$.
Furthermore, if $s'\in[0,s]$, then the error of the best $n$--term wavelet 
approximation to $u$ in the norm of $H^{s'}(\partial\Omega)$ satisfies
\begin{equation*}
	\sigma_n\!\left( u; \Psi^{\partial\Omega}, H^{s'}(\partial\Omega) \right) \lesssim n^{-\gamma/2}
	\qquad \text{for all} \qquad
	\gamma < 2 \cdot \left(1-\frac{s'}{s} \right) \cdot \min\{\rho,k-\rho,s\}.
\end{equation*}
\end{theorem}

\section{Adaptive wavelet methods for integral equations}\label{sec:AWEM}
Adaptive wavelet methods rely on an iterative solution method for 
the \emph{continuous boundary integral equation} \link{eq:generaleq} 
under consideration, expanded with respect to the wavelet basis. To 
this end, we renormalize the wavelet basis $\Psi$ w.r.t.\ the underlying
energy space. Then \link{eq:generaleq} is equivalent to the well--posed problem of finding $u=\Psi^{\partial\Omega}{\bf u}$ such that 
the \emph{infinite dimensional} system of linear equations
\begin{equation}	\label{eq:discrete problem}
  {\bf Su} = {\bf f}, \quad \text{where} \quad {\bf S}:=\distr{ S(\Psi^{\partial\Omega})}{\Psi^{\partial\Omega}} \quad \text{and} 
  	\quad {\bf f}:= \distr{f}{\Psi^{\partial\Omega}},
\end{equation}
holds. For approximately solving this infinite dimensional system of
linear equations, one has to perform matrix-vector multiplications by 
means of adaptive applications of the matrix ${\bf S}$ under consideration. 
The basic building blocks $\coarse$, $\apply$, $\rhs$, and $\solve$, which 
are needed to arrive at an adaptive algorithm of optimal complexity, have 
been introduced in \cite{CDD01,CDD02}. Our particular implementation 
is based on piecewise constant wavelets as outlined in \cite{DHS1,HU}, 
see also \cite{GantStev} for related results. In particular, we restrict the 
set of active wavelet functions to tree constraints which ensures the
method's efficient implementation. 
Notice that the piecewise constant wavelets we use here are discontinuous. This implies that the norm 
equivalence \link{equiv_sobolev_norm} only holds in the range 
$-1/2 < s < 1/2$. 
Nevertheless, this limitation in
the basis functions' smoothness does not change the rates of the
best $n$--term approximations in the range of Sobolev spaces 
which can be characterized. 

The specific adaptive algorithm we use has been proposed in \cite{G,GHS} 
and is similar to classical methods which consist of the following steps:
\begin{center}
  \boxed{\text{\sc Solve}}\quad$\longrightarrow$\quad
  \boxed{\text{\sc Estimate}}\quad$\longrightarrow$\quad
  \boxed{\text{\sc Mark}}\quad$\longrightarrow$\quad
  \boxed{\text{\sc Refine}}
\end{center}
For a given (finite) index set $\mathcal{T}\subset\nabla^{\partial
\Omega}$, we \emph{solve} the Galerkin system \link{eq:discrete 
problem} via ${\bf u}_{\mathcal{T}} = \solve[\mathcal{T}]$. 
Then we \emph{estimate} the (infinite) residuum 
${\bf r} := {\bf f}-{\bf S}{\bf u}_{\mathcal{T}}$ with sufficient accuracy $\delta>0$ 
by computing
\begin{equation*}
	{\bf r}_{\mathcal{T}'} = \rhs[\delta/2]-\apply[\delta/2,{\bf u}_{\mathcal{T}}]
\end{equation*}
relative to a finite index set $\mathcal{T}\subset\mathcal{T}'
\subset\nabla^{\partial\Omega}$ such that 
\begin{equation*}
  	\norm{ {\bf r}-{\bf r}_{\mathcal{T}'} }_2 \leq \delta.
\end{equation*}
Herein, $\rhs[\delta/2]$ produces a finitely supported 
approximation of the right-hand side with accuracy $\delta/2$ 
and $\apply[\delta/2,{\bf u}_{\mathcal{T}}]$ approximates the 
matrix-vector product ${\bf Su}_{\mathcal{T}}$ with
accuracy $\delta/2$. In order to have $\norm{ {\bf r}_{\mathcal{T}'} }_2$ 
proportional to $\norm{ {\bf r} }_2$, i.e.,
\begin{equation}\label{eq:equivalence}
  (1-\omega) \, \norm{ {\bf r}_{\mathcal{T}'} }_2 \leq \norm{ {\bf r} }_2
  	\leq (1+\omega)\, \norm{ {\bf r}_{\mathcal{T}'} }_2
\end{equation}
for fixed $0<\omega<1$, we apply the following 
iteration for some initial precision $\delta_{\mathrm{init}}$:
\begin{equation}\label{eq:grow}
\left.\begin{array}{l}
\text{set}\ \delta = \delta_{\mathrm{init}}\\[1ex]
\text{do}\\[1ex]
\qquad\text{set}\ \delta = \delta/2\\[1ex]
\qquad\text{calculate}\ {\bf r}_{\mathcal{T}'} = \rhs[\delta/2]
	-\apply[\delta/2,{\bf u}_{\mathcal{T}}]\\[1ex]
\text{until}\ \delta \leq \omega \, \norm{ {\bf r}_{\mathcal{T}'} }_2
\end{array}\qquad\right\}
\end{equation}
The until-clause $\delta \leq \omega \, \norm{ {\bf r}_{\mathcal{T}'} }_2$ 
causes that this iteration terminates when 
\link{eq:equivalence} holds.

The supporting index set $\mathcal{T}'$ of the 
approximate residuum ${\bf r}_{\mathcal{T}'}$ enlarges
the original index set $\mathcal{T}$ enough to ensure that
the Galerkin solution with respect to $\mathcal{T}'$ would 
reduce the error by a constant factor. Nevertheless, to 
control the complexity, we have to coarsen the index set 
$\mathcal{T}'$ such that
\begin{equation*}
  \norm{ {\bf r}_{\mathcal{T}''} }_2 \leq \theta \, \norm{ {\bf r}_{\mathcal{T}'} }_2
\end{equation*}
for fixed $0<\theta<1$ sufficiently small. This is done by calling
\begin{equation*}
  {\bf r}_{\mathcal{T}''} = \coarse[\theta,{\bf r}_{\mathcal{T}'}].
\end{equation*}
It combines the steps \emph{mark} and \emph{refine} since 
the new index set $\mathcal{T}''$ enlarges the original index 
set $\mathcal{T}$ which corresponds to mesh refinement. 
We emphasize that $\mathcal{T}''$ is still large enough to 
guarantee the convergence of the algorithm when starting 
the procedure again with $\mathcal{T}:=\mathcal{T}''$. For 
all the details of the particular implementation, we refer the 
reader to \cite{Manu}.

\section{Numerical results}\label{sec:numerix}
\subsection{Right-hand side with point singularity}
We will present results for the Laplace equation solved by 
the second kind Fredholm integral equation \link{DoubleLayerProb}
for the double layer potential operator \link{DLP}. We choose Fichera's 
vertex as domain under consideration, i.e., $\Omega := (0,1)^3\setminus (0,0.5]^3$.
Its surface is parametrized by 12 patches in accordance with \autoref{sect:boundaries}. We consider non-smooth Dirichlet data of the form
\begin{equation*}
	g(x) := |x-(0.5,0.5,0.5)|^{-\alpha}
\end{equation*}
for $\alpha = 0.5$ and $\alpha = 0.75$. These Dirichlet data
admit a point singularity in the reentrant corner $(0.5, 0.5, 0.5)$
of the Fichera vertex. 

According to the subsequent \autoref{thm:example}, for these right-hand sides we can expect a rate of convergence of least $n^{-(1-\alpha)}$ when 
using an optimal adaptive scheme.  In contrast, we 
expect only half the rate, i.e., $n^{-(1-\alpha)/2}$, when using uniform 
refinement.

\begin{theorem}\label{thm:example}
Let $\partial\Omega\subset\R^3$ denote a Lipschitz surface according to \autoref{sect:boundaries}, assume $\nu$ to be one of its vertices, and let
$1/2 \leq \alpha < 1$. Moreover, denote by $u$ the unique solution 
to \link{DoubleLayerProb} with right-hand side given by 
\begin{equation*}
g(x):=\abs{x-\nu}^{-\alpha}, \quad x\in\partial\Omega.
\end{equation*}
Then the error of the best $n$--term wavelet approximation to $u$ w.r.t.\ $L_2(\partial\Omega)$ converges at least at a rate of $n^{-(\alpha-1)}$, while the corresponding rate for uniform approximation is limited to $n^{-(1-\alpha)/2}$.
\end{theorem}
\begin{proof}
From \autoref{prop:regularity} in the \aref{sect:appendix} we have that $\sup\{s>0 \sep g\in H^s(\partial\Omega)\}=1-\alpha$ and thus \autoref{prop:Verchota} implies that the same upper bound holds for the solution $u$. 
In turn, this proves the limitation for the rate of convergence using uniform refinement. On the other hand, \autoref{prop:regularity} also yields that we can apply \autoref{thm:doublelayer} with $k:=1$ and $s:=\rho:=1-\alpha-\epsilon$, where $\epsilon>0$ can be chosen arbitrarily small, as well as $s':=0$. This shows that the best $n$--term rate is at least twice as high which completes the proof.
\end{proof}

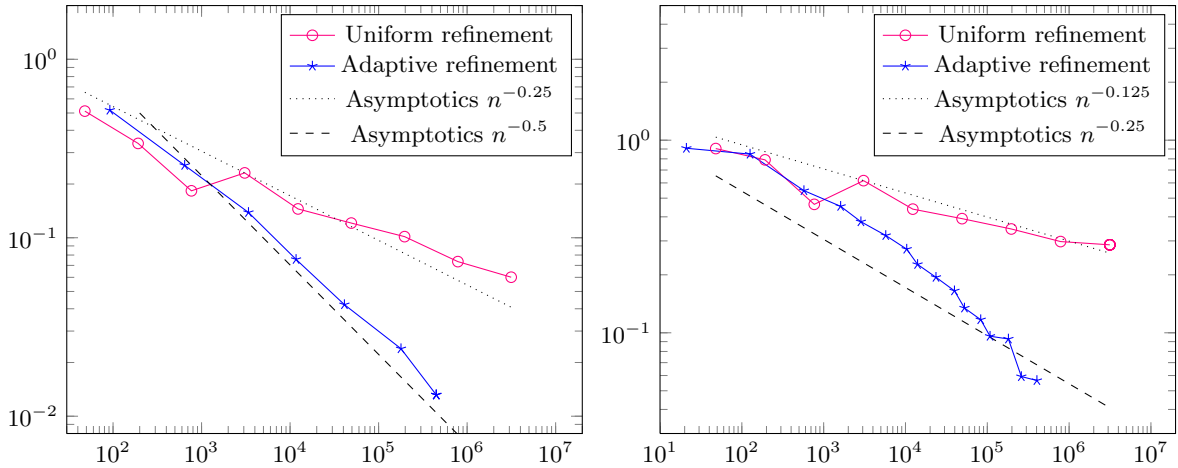
\begin{figure}[ht!]
\begin{center}
\pgfplotstableread{
48 0.5123365119 92 0.5186398723 
192 0.3377463262 644 0.2544818477 
768 0.1832319011 3389 0.1386063301 
3072 0.2308786237 11685 0.0756392232 
12288 0.1450360075 41098 0.0421897889 
49152 0.1208571488 179059 0.0238976208 
196608 0.1014540592 448234 0.0131669205 
786432 0.0734999322 448234 0.0131669205 
3145728 0.0601079572 448234 0.0131669205 
}\datatable
{\footnotesize\begin{tikzpicture}[scale=1]
\begin{loglogaxis}[xmin=30, xmax=20000000, ymin=0.008 , ymax=2]

\addplot [color=ManuRed,mark=o] table[x index=0, y index=1] {\datatable};
\addplot [color=blue, mark=star] table[x index=2, y index=3] {\datatable};

\addplot [color=black,dotted]  plot coordinates {(48, 0.653) (3145728, 0.0408)};
\addplot[color=black,dashed] plot coordinates {(200, 0.5) (1000000, 0.007)};

\legend{Uniform refinement\\ Adaptive refinement\\Asymptotics $n^{-0.25}$\\Asymptotics $n^{-0.5}$\\}
\end{loglogaxis}
\end{tikzpicture}}
\pgfplotstableread{
48 0.9047392763 21 0.9080435412
192 0.7905560541 126 0.8465289688
768 0.4643444653 574 0.5472434816
3072 0.6165508005 1612 0.4535397633
12288 0.4385794108 2853 0.3774309128
49152 0.3914645320 5739 0.3201015369
196608 0.3459198310 10355 0.2718107127
786432 0.2976181170 14054 0.2264756227
3145728 0.2862975718 23591 0.1941859271
3145728 0.2862975718 39702 0.1651967132
3145728 0.2862975718 52590 0.1345236012
3145728 0.2862975718 83091 0.1169779225
3145728 0.2862975718 108016 0.0959835508
3145728 0.2862975718 181216 0.0928693735
3145728 0.2862975718 260604 0.0592611440
3145728 0.2862975718 404049 0.0565484260
}\datatable
{\footnotesize\begin{tikzpicture}[scale=1]
\begin{loglogaxis}[xmin=10, xmax=20000000, ymin=0.03 , ymax=5]

\addplot [color=ManuRed,mark=o] table[x index=0, y index=1] {\datatable};
\addplot [color=blue, mark=star] table[x index=2, y index=3] {\datatable};

\addplot [color=black,dotted]  plot coordinates {(48, 1.037) (3145728, 0.2592)};
\addplot [color=black,dashed]  plot coordinates {(48, 0.653) (3145728, 0.0408)};

\legend{Uniform refinement\\ Adaptive refinement
\\Asymptotics $n^{-0.125}$\\Asymptotics $n^{-0.25}$\\}
\end{loglogaxis}
\end{tikzpicture}}
\caption{Energy norm of the residual vector for adaptive and uniform refinement for $\alpha = 0.5$
(left) and $\alpha = 0.75$ (right).\label{fig:convFich}}
\end{center}
\end{figure}

In \autoref{fig:convFich}, the observed convergence rates are found.
We plotted the energy norm of the residual vector for both, uniform and 
adaptive refinement, into a log-log plot. Notice that the results of 
the uniform refinement are produced by the same adaptive wavelet 
method via enforcing a uniform (hence non-adaptive) refinement in the step {\sc Refine}.

For $\alpha = 0.5$ and uniform refinement, the left plot in \autoref{fig:convFich} shows a rate that seems to be even slightly worse than $n^{-(1-\alpha)/2} = n^{-0.25}$.
In case of adaptive refinement, we obtain a rate of $n^{-(1-\alpha)}
= n^{-0.5}$, which is exactly what we expect. Another observation can 
be made by comparing the number of degrees of freedom which 
are necessary in order to compute the approximate density for 
uniform refinement and for adaptive refinement. The norm of the 
residual is about $6 \cdot 10^{-2}$ for uniform refinement with 
more than $3$ million degrees of freedom. For adaptive refinement, 
we obtain a norm of the residual of about $4.2 \cdot 10^{-2}$ already 
for approximately $40\,000$ degrees of freedom, which is quite 
impressive.

\begin{figure}[ht!]
\begin{center}
\includegraphics[scale=0.28]{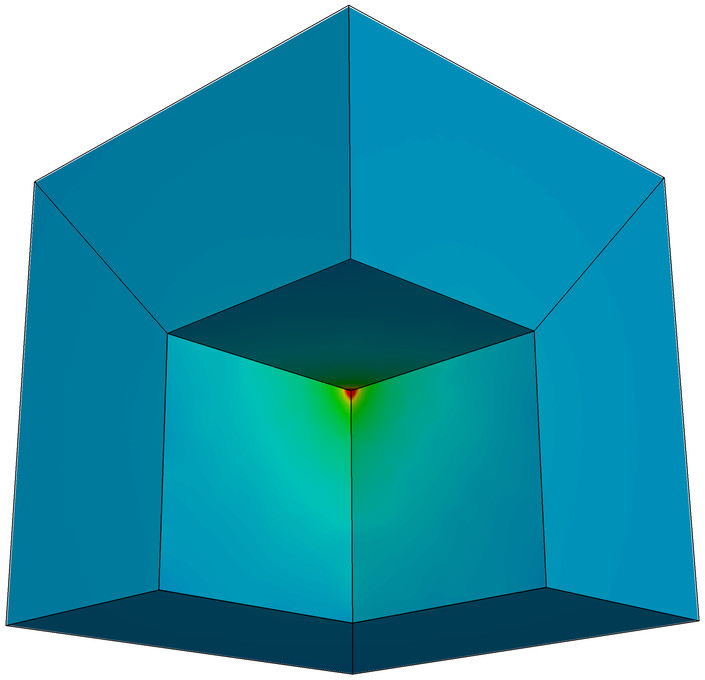}\quad
\includegraphics[scale=0.28]{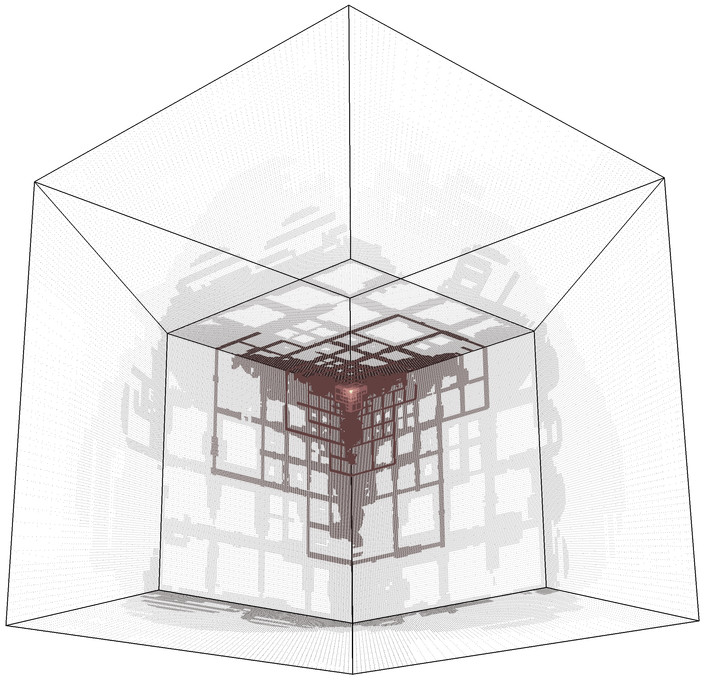}
\caption{Approximate density $u$ (left) and associated 
adaptive refinement (right) for $\alpha = 0.5$.\label{fig:Fich50}}
\end{center}
\end{figure}

We find the approximate density in the left image of \autoref{fig:Fich50}, while the refinement produced by the adaptive wavelet scheme is shown right next to it. The refinement is visualized by 
plotting the indices of all active wavelets. It is clearly visible that the 
adaptive method refines towards reentrant corner, where the right-hand 
side has its singularity. We also observe an interesting pattern in form 
of a grid in the refinement around the corner. This artefact comes 
from the large support of the wavelets. 

\begin{figure}[ht!]
\begin{center}
\includegraphics[scale=0.28]{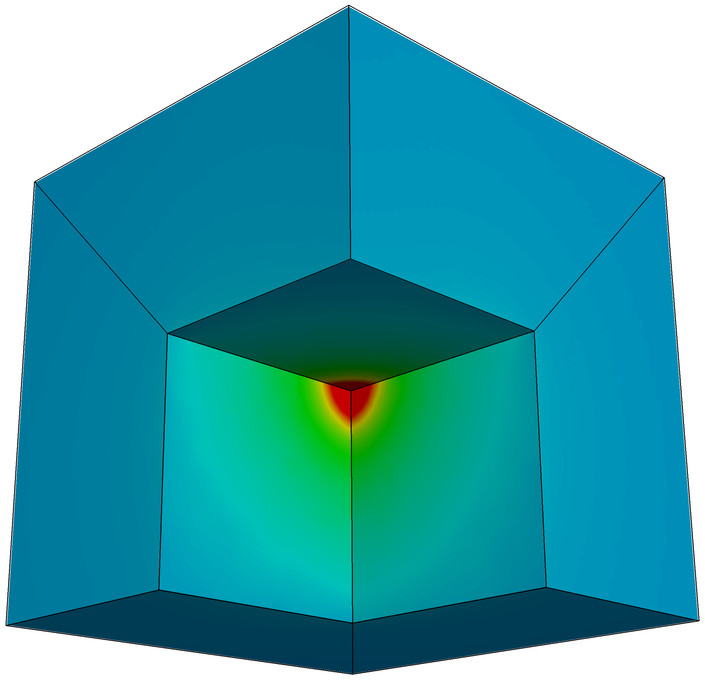}\quad
\includegraphics[scale=0.28]{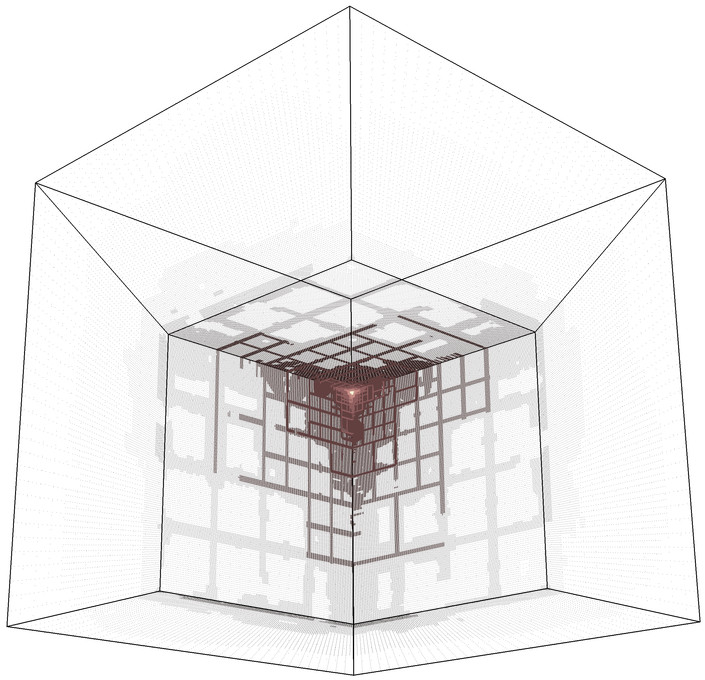}
\caption{Approximate density $u$ (left) and associated 
adaptive refinement (right) for $\alpha = 0.75$.\label{fig:Fich75}}
\end{center}
\end{figure}

Next, let us discuss the numerical results for $\alpha = 0.75$. Due to \autoref{thm:example}, for this 
choice, we expect the adaptive wavelet method to converge 
at a rate of at least $n^{-(1-\alpha)} = n^{-0.25}$. 
In contrast, for uniform refinement, 
we again only expect at most half the rate, i.e., $n^{-(1-\alpha)/2} = n^{-0.125}$. 
Indeed, it can be seen in the right plot of \autoref{fig:convFich} that the adaptive wavelet scheme converges 
at the expected rate or even slightly better. In comparison, we 
observe the reduced rate $n^{-0.125}$ for uniform refinement.
By comparing the norm of the residuals for both strategies, we 
again confirm the superiority of the adaptive code. In order for 
uniform refinement to produce an error of $2.9 \cdot 10^{-1}$, it 
needs more that $3$ million degrees of freedom, whereas the adaptive 
code produces an error of $2.7 \cdot 10^{-1}$ with less than 
$15 \, 000$ degrees of freedom. 
The approximate density, as well as the adaptive refinement, for this example finally are found in \autoref{fig:Fich75}.

\subsection{Cartoon function as right-hand side}
In the last example, the cube $\Omega := (-1,1)^3$ is chosen 
as domain, the boundary of which is represented by 6 patches. 
The right-hand side under consideration is a cartoon function, 
namely
\begin{equation*}
  g(x) 
  := \begin{cases} 
  		1 & \text{if}\ \abs{x-(0,0,1)}^2\le\frac{1}{2},\\
  		0, & \text{elsewhere}.
  	\end{cases}
\end{equation*}

Cartoon functions have been studied in \cite{Don}. 
They can be approximated adaptively in an isotropic setting at the rate 
of $n^{-0.5}$. Since, however, cartoon functions $g$ have a jump 
discontinuity, the best we can expect is $g\in H^{1/2-\epsilon}
(\partial\Omega)$ with $\epsilon>0$ being arbitrary small. 
Therefore, \autoref{prop:Verchota} implies that the Sobolev regularity 
of the exact solution $u\in L_2(\partial\Omega)$ to the second kind 
Fredholm integral equation \link{DoubleLayerProb} for the double 
layer potential operator~\link{DLP} with right-hand side $g$ is also 
bounded by $1/2$. Consequently, for the uniform code, we would 
expect the approximation rate $n^{-0.25}$. This is in perfect 
accordance with our numerical experiment; see \autoref{fig:convCube}.
Moreover, also in this example the observed rate for adaptive 
refinement is twice as large, i.e., the norm of the residuum 
decays at a rate of $n^{-0.5}$. Indeed, as seen 
in the left image of \autoref{fig:Cube}, the approximate density 
is basically also a cartoon function. This issues from the fact
that the kernel of the double layer operator is zero on
a plane patch since $x-y$ is perpendicular to the normal $\eta(y)$.
The right image of \autoref{fig:Cube} illustrates the corresponding 
refinement produced by the adaptive algorithm. It is clearly seen that 
refinement mainly takes place at the jump of the right-hand side.

\pgfplotstableread{
24 0.1363953699 32 0.1364189637
96 0.1347929985 147 0.0687532832 
384 0.0686609961 738 0.0369364185 
1536 0.0482202337 3129 0.0195492698
6144 0.0318406504 8207 0.0102049593
24576 0.0277205991 29029 0.0073730154
98304 0.0188861182 77311 0.0046253998 
393216 0.0137424768 185680 0.0029000619
393216 0.0137424768 574088 0.0021026280
}\datatable
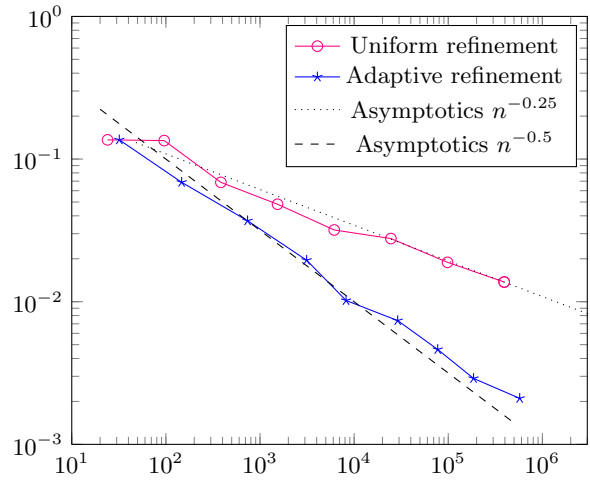
\begin{figure}[t!]
\begin{center}
{\footnotesize
\begin{tikzpicture}[scale=1]
\begin{loglogaxis}[xmin=10, xmax=3000000, ymin=0.001 , ymax=1]
\addplot [color=ManuRed,mark=o] table[x index=0, y index=1] {\datatable};
\addplot [color=blue, mark=star] table[x index=2, y index=3] {\datatable};
\addplot [color=black,dotted]  plot coordinates {(48, 0.1306) (3145728, 0.00816)};
\addplot[color=black,dashed] plot coordinates {(20, 0.2236) (500000, 0.001414)};
\legend{Uniform refinement\\Adaptive refinement\\Asymptotics $n^{-0.25}$\\Asymptotics $n^{-0.5}$\\}
\end{loglogaxis}
\end{tikzpicture}}
\caption{Energy norm of the residual for adaptive and uniform refinement 
for the cartoon right-hand side.\label{fig:convCube}}
\end{center}
\end{figure}

\begin{figure}[ht!]
\begin{center}
\includegraphics[scale=0.28]{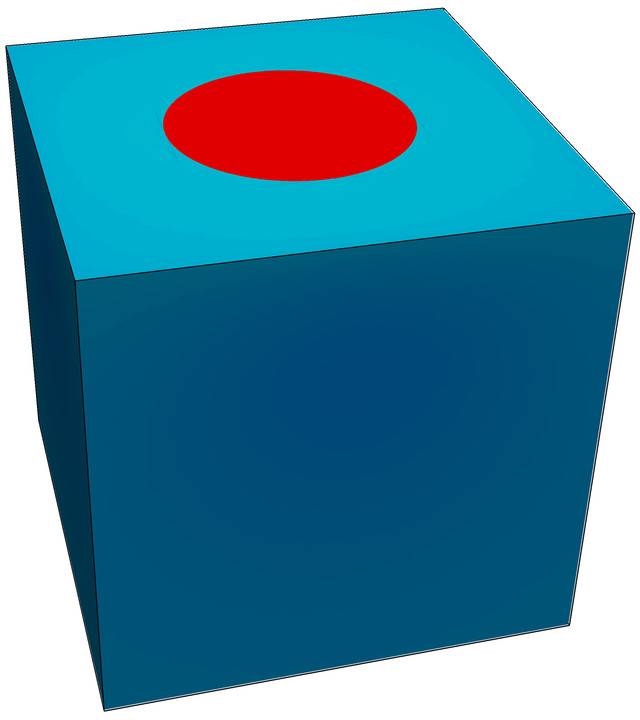}\quad
\includegraphics[scale=0.28]{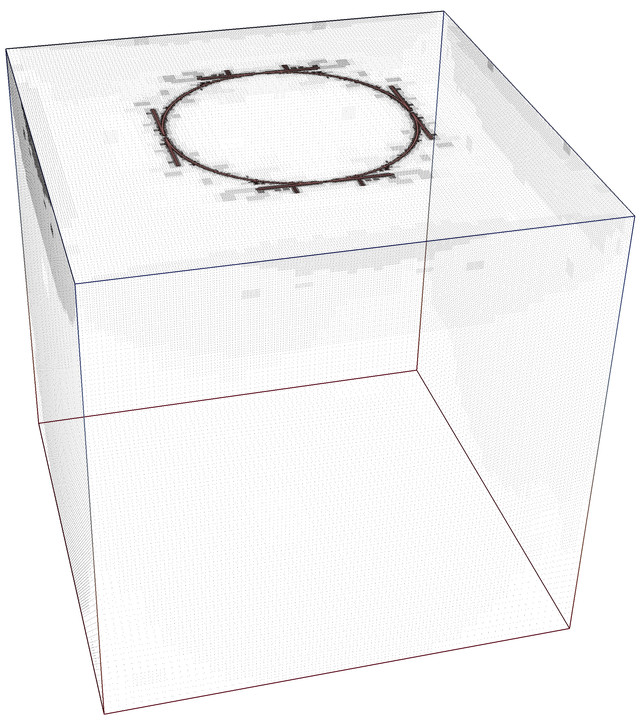}
\caption{Approximate density $u$ (left) and 
associated refinement (right) for the cartoon right-hand 
side on the cube.\label{fig:Cube}}
\end{center}
\end{figure}

\section{Conclusion}\label{sec:conclusion}
In this article, we considered the adaptive solution of boundary
integral equations by using wavelet bases. We have presented 
the regularity theory which is available and presented numerical
results which are in good agreement with the theoretical predictions.
In conclusion, we have seen that the adaptive wavelet BEM is able 
to produce optimal convergence rates, which, for the examples 
under consideration, are twice as high as for methods based on
uniform refinement.

\begin{appendix}
\phantomsection
\addcontentsline{toc}{section}{Appendix}
\section*{Appendix}\label{sect:appendix}
\setcounter{section}{1}

Here we calculate the weighted and unweighted Sobolev regularity of a point singularity. To do so, we need the following auxiliary estimate:

\begin{lemma}\label{lem:lowerbound}
Let $d\in\N$, $\alpha>0$ and $M>2$. Then for all $x,h\in\R^d$ with $x_j, h_j\geq 0$ for every $j=1,\ldots,d$ and $0<\abs{x}\leq \abs{h}/M$ we have
\begin{align*}
\abs{\abs{x}^{-\alpha} - \abs{x+h}^{-\alpha}} \geq (M^{\alpha} - 2^{\alpha})\, \abs{h}^{-\alpha}
\end{align*}
\end{lemma}
\begin{proof}
Since $x_j,h_j\geq 0$ for all $j=1,\ldots,d$ we have $\abs{x}\leq\abs{x+h}$ which yields $\abs{x}^{-\alpha} \geq \abs{x+h}^{-\alpha}$ because of $\alpha>0$.
This gives
\begin{align*}
\abs{h}^{\alpha} \cdot \abs{\abs{x}^{-\alpha}- \abs{x+h}^{-\alpha}}
&= \left( \frac{\abs{x}}{\abs{h}} \right)^{-\alpha} - \left( \frac{\abs{x+h}}{\abs{h}} \right)^{-\alpha} 
\geq M^{\alpha} - \left( \frac{\abs{h}}{\abs{x+h}} \right)^{\alpha} 
\end{align*}
since $\abs{x} \leq \abs{h}/M$ implies $\abs{x}^{\alpha} \leq (\abs{h}/M)^{\alpha}$ and $(\abs{x}/\abs{h})^{-\alpha} \geq M^{\alpha}$.
Now the triangle inequality combined with $\abs{x}/ \abs{x+h} \leq 1$ shows that the latter quantity is lower bounded by 
\begin{align*}
M^{\alpha} - \left( \frac{\abs{h}}{\abs{x+h}} \right)^{\alpha} 
&\geq M^{\alpha} - \left( \frac{\abs{x+h}+\abs{x}}{\abs{x+h}} \right)^{\alpha}
= M^{\alpha} - \left( 1+\frac{\abs{x}}{\abs{x+h}} \right)^{\alpha}
\geq M^{\alpha} - 2^{\alpha}
\end{align*}
and the proof is complete.
\end{proof}

Now the regularity result reads as follows.

\begin{prop}\label{prop:regularity}
Let $\partial\Omega$ denote a Lipschitz surface according to \autoref{sect:boundaries} and assume $\nu$ to be one of its vertices.
Moreover, given $1/2 \leq \alpha < 1$, 
consider the restriction of 
\begin{equation*}
g(x):=\abs{x-\nu}^{-\alpha}, \quad x\in\R^3,
\end{equation*}
to $\partial\Omega$.
Then
\begin{align*}
g\in H^s(\partial\Omega) \cap X_{\rho}^1(\partial\Omega) \qquad \text{for all} \quad 0\leq s < 1-\alpha \quad \text{and all} \quad 0\leq \rho < 1-\alpha,
\end{align*}
but $g\notin H^{s}(\partial\Omega)$ if $s \geq 1-\alpha$.
\end{prop}
\begin{proof}
\emph{Step 1.} Here we show that (the restriction of) $g$ is contained in $H^s(\partial\Omega)$. Note that $\alpha \geq 1/2$ implies $0\leq s<1/2$, i.e., it suffices to prove that $g$ belongs to each Sobolev space w.r.t.\ every single patch $F$ that is part of the description of $\partial\Omega$. According to the structure of $g$ it is obvious that $g\in H^1(F) \hookrightarrow H^s(F)$ for all patches $F$ which do not contain the critical vertex $\nu$ since $g$ is bounded and smooth on the closure of these patches.

By construction each of the remaining patches $F$ (with $\nu\in\overline{F}$) are contained in some face $\Gamma$ of the infinite tangent cone $\mathcal{C}$ subordinate to the vertex $\nu$.
W.l.o.g.\ we can assume that $\nu=0$ and 
\begin{align}\label{eq:gamma}
	\Gamma =\{y=(r\,\cos(\phi),r\,\sin(\phi))\in \R^2 \sep r>0,\, \phi\in(0,\gamma) \}
\end{align}
with some opening angle $\gamma \in(0,2\pi)$. Then the restriction of $g$ to $F$ takes the form
\begin{equation*}
	g(y)=r^{-\alpha}, \qquad y\in F\subset \Gamma.
\end{equation*}

For $s\geq 0$ and $d\in\N$ it is well-known that $H^s(M)=F^s_{2,2}(M)=B^{s}_{2,2}(M)$, where $M$ can be any (bounded or unbounded) Lipschitz domain in $\R^d$ or $\R^d$ itself. Here $B^s_{p,q}$ and $F^s_{p,q}$ (where $0<p,q<\infty$ and $s\in\R$) denote the scales of \emph{classical Besov and Triebel-Lizorkin spaces}, respectively; see, e.g., \cite[Subsection~1.11.1]{T06} for precise definitions. Hence, in order to prove the claim, it suffices to show that $g\in F^s_{2,2}(F)$.
According to the definition of this space it is enough to find some $\tilde{g}\in F^s_{2,2}(\R^2)$ which extends $g$ from $F$ to the whole of $\R^2$.
To this end, we choose a sufficiently smooth cut-off function $\zeta\colon \R\to [0,1]$ with 
\begin{equation*}
	\zeta(r)
	=\begin{cases}
		1 & \quad \text{if }\; 0\leq r < R,\\
		0 & \quad \text{if }\; 2R\leq r,
	\end{cases}
\end{equation*}
and $R>0$ large enough such that $\zeta(\abs{y})\equiv 1$ for all $y\in F$. Then 
\begin{align*}
	\tilde{g}(y):=\abs{y}^{-\alpha}\,\zeta(\abs{y}), \qquad y\in\R^2,
\end{align*}
defines an extention of $g$, i.e., it 
satisfies $\tilde{g}\big|_{F}=g$. 
Hence, it only remains to prove the membership of $\tilde{g}$ in $F^s_{2,2}(\R^2)$.
Fortunately, $\tilde{g}$ with $\alpha \neq 0$ coincides with the function $f_{-\alpha,0}$ studied in \cite[Subsection~2.3.1]{RS96} and there it is shown that for $s>2\cdot\max\{0,1/p-1\}$ and $0< p,q < \infty$ 
we indeed have $\tilde{g}=f_{-\alpha,0}\in F^s_{p,q}(\R^2)$ if and only if $s<2/p - \alpha$. 
This particularly shows that $0<\alpha<1$ implies $\tilde{g}\in F^s_{2,2}(\R^2)$ for all $0<s<1-\alpha$. But then a simple embedding shows that the case $s=0$ is covered as well. So, the proof of $g\in H^s(\partial\Omega)$ for all $0\leq s < 1-\alpha$ is complete.

\emph{Step 2.} We turn to the proof of the membership of $g$ in $X^1_\rho(\partial\Omega)$. Here we have to show that $\varphi_n g \in X^1_\rho(\partial\mathcal{C}_n)$, $n=1,\ldots,N$, where $\partial\mathcal{C}_n$ is the boundary of the infinite tangent cone subordinate to the $n$th vertex $\nu_n$ and $\varphi_n$ is a smooth cut-off function with support in a neighborhood of $\nu_n$. For this purpose, we have to bound the norm $\norm{\varphi_n g \sep X^1_{\rho}(\partial\Co_n)}$ given by \link{eq:normCn} with $f_n$ replaced by $\varphi_n g$ and $k=1$. Due to Step 1 we have $g\in L_2(\partial\Omega)$ and hence $\norm{\varphi_n g \sep L_2(\partial\mathcal{C}_n)}<\infty$ for all $n=1,\ldots,N$. Thus, it remains to estimate
\begin{equation*}
	\sum_{(\beta_r,\beta_\phi)\in\{(1,0),\, (0,1)\}} 
		\norm{ \left(1+\frac{1}{r}\right)^{\rho} \,r^{\beta_r}\, q^{1-\rho} \, \frac{\partial}{\partial \phi^{\beta_\phi}\,\partial r^{\beta_r}}  (\varphi_n g)_{n,t} \sep L_2\!\left(\Gamma^{n,t}\right) }
\end{equation*}
for all vertices $\nu_n$, $n=1,\ldots,N$, and each face $\Gamma^{n,t}$, $t=1,\ldots,T_n$, of their associated tangent cones. So let $n$ and $t$ be fixed. In order to simplify notation we can again assume that $\Gamma=\Gamma^{n,t}$ is given by~\link{eq:gamma}. If $\nu_n\neq\nu$, then 
\begin{align*}
G:=\frac{\partial}{\partial \phi^{\beta_\phi}\,\partial r^{\beta_r}}  \varphi_n g
\end{align*}
is smooth on $\overline{\Gamma}$ and thus uniformly bounded. Due to the compact support of $\varphi_n g$ and the fact that $q(\phi)\in(0,\pi)$, we obtain
\begin{align*}
\norm{ \left(1+\frac{1}{r}\right)^{\rho} \, r^{\beta_r}\, q^{1-\rho}\,G 
	\sep L_2\!\left(\Gamma\right) }^2
	&\leq \norm{G}_\infty^2 \, \int_0^\gamma q(\phi)^{2(1-\rho)} \, \int_0^R \abs{\left(1+\frac{1}{r}\right)^{\rho} r^{\beta_r} }^2 \, r \d r \d \phi \\
	&\lesssim \int_0^R \left(1+r\right)^{2\rho} r^{1-2\rho+2\beta_r} \d r < \infty
\end{align*}
for all $\beta=(\beta_r,\beta_\phi)\in\N_0^2$ provided that $0\leq \rho\leq 1/2$.

Now, if $\nu_n=\nu$, let us define $\Gamma_0:=\Gamma \cap \{y\in\R^2\sep \abs{y}<R_0\}$ with $R_0 \in (0,R)$ chosen small enough such that $\varphi_n\equiv 1$ in a neighborhood of $\Gamma_0$. 
Then
\begin{align*}
	\norm{ \left(1+\frac{1}{r}\right)^{\rho} \,r^{\beta_r}\, q^{1-\rho} \, \frac{\partial}{\partial \phi^{\beta_\phi}\,\partial r^{\beta_r}} \varphi_n g	\sep L_2\!\left(\Gamma\right) }^2 
&=\norm{ \left(1+\frac{1}{r}\right)^{\rho} \,r^{\beta_r}\, q^{1-\rho} \, \frac{\partial}{\partial \phi^{\beta_\phi}\,\partial r^{\beta_r}} g
	\sep L_2\!\left(\Gamma_0\right) }^2 \\
&\quad + \norm{ \left(1+\frac{1}{r}\right)^{\rho} \,r^{\beta_r}\, q^{1-\rho} \, \frac{\partial}{\partial \phi^{\beta_\phi}\,\partial r^{\beta_r}} \varphi_n g
	\sep L_2\!\left(\Gamma\setminus\Gamma_0\right) }^2.
\end{align*}
Note that for $\Gamma\setminus \Gamma_0$ the same arguments as above apply. Moreover, if $\beta=(\beta_r,\beta_\phi)=(0,1)$, then the first summand vanishes due to the rotational invariance of $g$, i.e., $\partial g/\partial \phi \equiv 0$ as $g$ does not depend on the angular variable $\phi$.
So we are left with bounding the first term for $\beta=(1,0)$. Here we have
\begin{align*}
&\int_0^\gamma \int_0^{R_0} \abs{ \left(1+\frac{1}{r}\right)^{\rho} \,r \, q(\phi)^{1-\rho} \, \frac{\partial}{\partial r} r^{-\alpha} }^2 \, r \d r \d \phi \\
&\qquad =\int_0^\gamma q(\phi)^{2(1-\rho)} \d \phi\, \int_0^{R_0} \left(1+\frac{1}{r}\right)^{2\rho} \,r^3  \, \abs{-\alpha \,r^{-\alpha-1} }^2 \d r \\
&\qquad \lesssim \int_0^{R_0} \left(1+r\right)^{2\rho} \, r^{1-2\alpha-2\rho} \d r < \infty
\end{align*}
provided that $0\leq \rho \leq 1$ and $1-2\alpha-2\rho>-1$ which particularly holds under the given assumptions $0\leq \rho < 1-\alpha$ and $1/2\leq \alpha<1$. 

\emph{Step 3.} Finally, we show that $g\notin H^{s}(\partial\Omega)$ if $s \geq 1-\alpha$. Since $0<1-\alpha \leq 1/2$ it suffices to prove $g\notin H^{1-\alpha}(F)=B_{2,2}^{1-\alpha}(F)$ for some patch $F\subset \partial\Omega$. To do so, let $F$ denote a patch with $\nu\in \overline{F}$ and let $\Gamma$ be the face of the corresponding tangent cone associated to $\nu$ which contains $F$. W.l.o.g.\ we again assume that $\nu=0$ and that $\Gamma$ is given by \link{eq:gamma}.
Then we can choose $T>0$ small enough and $M>2$ large enough such that
for all $h \in H_{T,M}:=\{\in\R^2 \sep h_1,h_2\geq 0,\, \abs{h}\leq T\}$ we have 
\begin{equation*}
\Gamma_h:=\Gamma\cap \left\{y\in\R^2 \sep x_1,x_2\geq0, \, \abs{y}< \frac{\abs{h}}{M} \right\} \subseteq F_h:=\{y\in F \;|\; y+h \in F\} \subseteq F.
\end{equation*} 
Then \autoref{lem:lowerbound} yields that for all $h \in H_{T,M}$ the first order finite difference $\Delta_h g:=g(\cdot+h)-g$ satisfies
\begin{align*}
\norm{ \Delta_h g \sep L_2(\Gamma_h)}^2
&=\int_{\Gamma_h} \abs{\Delta_h g(y)}^2 \d y \\
&=\int_{\Gamma_h} \abs{\abs{y}^{-\alpha} - \abs{y+h}^{-\alpha}}^2 \d y \\
&\geq (M^{\alpha} - 2^{\alpha})^2 \,\abs{h}^{-2\alpha} \int_0^{\min\{\gamma,\, \pi/2\}} \int_{0}^{\abs{h}/M} r \d r \d \phi \\
&\sim_{M,\alpha,\gamma} \abs{h}^{2-2\alpha}
\end{align*}
and hence for the modulus of smoothness it holds
\begin{align*}
\omega(g,t,F)_2
:= \sup_{\substack{h\in\R^2\\\abs{h}\leq t}} \norm{ \Delta_h g \sep L_2(F_h)} 
\geq \sup_{\substack{h\in\R^2\\h_1,h_2\geq 0, \, \abs{h}\leq t}} \norm{ \Delta_h g \sep L_2(\Gamma_h)} 
\gtrsim t^{1-\alpha} \qquad \text{for all} \quad 0<t<T,
\end{align*}
where we used that $1-\alpha > 0$. This shows
\begin{align*}
\norm{g \sep H^{1-\alpha}(F)} \sim \norm{g \sep B_{2,2}^{1-\alpha}(F)} 
:= \int_{0}^\infty \left[ t^{-(1-\alpha)} \, \omega(g,t,F)_2 \right]^2 \frac{\d t}{t}
\gtrsim \int_{0}^T \frac{\d t}{t} = \infty
\end{align*}
and thus $g\notin H^{s}(\partial\Omega)$ if $s \geq 1-\alpha$.
\end{proof}
\end{appendix}

\phantomsection
\addcontentsline{toc}{section}{References}
\bibliographystyle{is-abbrv}

\begin{thebibliography}{10}
\setlength{\parskip}{-2pt}

\bibitem{Beb1}
M.~Bebendorf.
\newblock Approximation of boundary element matrices.
\newblock {\em Numer.\ Math.}, 86:\penalty0 565--589, 2000.

\bibitem{CTU99}
C.~Canuto, A.~Tabacco, and K.~Urban.
\newblock The wavelet element method part~{I}: {C}onstruction and analysis.
\newblock {\em Appl.\ Comput.\ Harmon.\ Anal.}, 6\penalty0 (1):\penalty0 1--52,
  1999.

\bibitem{CTU00}
C.~Canuto, A.~Tabacco, and K.~Urban.
\newblock The wavelet element method part~{II}: {R}ealization and additional
  features in {2D} and {3D}.
\newblock {\em Appl.\ Comput.\ Harmon.\ Anal.}, 8\penalty0 (2):\penalty0
  123--165, 2000.

\bibitem{CDD01}
A.~Cohen, W.~Dahmen, and R.~A. DeVore.
\newblock {A}daptive wavelet methods for elliptic operator equations:
  {C}onvergence rates.
\newblock {\em Math.\ Comp.}, 70:\penalty0 27--75, 2001.

\bibitem{CDD02}
A.~Cohen, W.~Dahmen, and R.~A. DeVore.
\newblock {A}daptive wavelet methods {II}: {B}eyond the elliptic case.
\newblock {\em Found.\ Comput.\ Math.}, 2\penalty0 (3):\penalty0 203--245,
  2002.

\bibitem{CM00}
A.~Cohen and R.~Masson.
\newblock Wavelet adaptive method for second order elliptic problems:
  {B}oundary conditions and domain decomposition.
\newblock {\em Numer.\ Math.}, 86\penalty0 (2):\penalty0 193--238, 2000.

\bibitem{DK87}
B.~E. Dahlberg and C.~E. Kenig.
\newblock Hardy spaces and the {N}eumann problem in {$L^p$} for {L}aplace's
  equation in {L}ipschitz domains.
\newblock {\em {Ann. Math. (2)}}, 125:\penalty0 437--465, 1987.

\bibitem{DDD97}
S.~Dahlke, W.~Dahmen, and R.~A. Devore.
\newblock {N}onlinear approximation and adaptive techniques for solving
  elliptic operator equations.
\newblock In W.~Dahmen, A.~Kurdila, and P.~Oswald, editors, {\em {M}ultsicale
  {W}avelet {M}ethods for Partial Differential Equations}, pages 237--283, San
  Diego, 1997. Academic Press.

\bibitem{DD97}
S.~Dahlke and R.~A. DeVore.
\newblock Besov regularity for elliptic boundary value problems.
\newblock {\em Comm.\ Partial Differential Equations}, 22\penalty0
  (1-2):\penalty0 1--16, 1997.

\bibitem{DNS06}
S.~Dahlke, E.~Novak, and W.~Sickel.
\newblock Optimal approximation of elliptic problems by linear and nonlinear
  mappings {II}.
\newblock {\em J.~Complexity}, 22\penalty0 (4):\penalty0 549--603, 2006.

\bibitem{DahWei2015}
S.~Dahlke and M.~Weimar.
\newblock Besov regularity for operator equations on patchwise smooth
  manifolds.
\newblock {\em Found.\ Comput.\ Math.}, 15\penalty0 (6):\penalty0 1533--1569,
  2015.

\bibitem{DHS07}
W.~Dahmen, H.~Harbrecht, and R.~Schneider.
\newblock {A}daptive methods for boundary integral equations: {C}omplexity and
  convergence estimates.
\newblock {\em Math.\ Comp.}, 76:\penalty0 1243--1274, 1998.

\bibitem{DHS1}
W.~Dahmen, H.~Harbrecht, and R.~Schneider.
\newblock Compression techniques for boundary integral equations.
  {A}symptotically optimal complexity estimates.
\newblock {\em SIAM J.~Numer.\ Anal.}, 43\penalty0 (6):\penalty0 2251--2271,
  2006.

\bibitem{DS99}
W.~Dahmen and R.~Schneider.
\newblock Composite wavelet bases for operator equations.
\newblock {\em Math.\ Comp.}, 68:\penalty0 1533--1567, 1999.

\bibitem{D98}
R.~A. DeVore.
\newblock Nonlinear approximation.
\newblock {\em Acta Numer.}, 7:\penalty0 51--150, 1998.

\bibitem{DJP92}
R.~A. DeVore, B.~Jawerth, and V.~Popov.
\newblock Compression of wavelet decompositions.
\newblock {\em Am.~J.~Math.}, 114\penalty0 (4):\penalty0 737--785, 1992.

\bibitem{Don}
D.~Donoho.
\newblock Sparse components of images and optimal atomic decomposition.
\newblock {\em Constr.\ Approx.}, 17:\penalty0 353--382, 2001.

\bibitem{E92}
J.~Elschner.
\newblock The double layer potential operator over polyhedral domains {I}:
  {S}olvability in weighted {S}obolev spaces.
\newblock {\em App.\ Anal.}, 45:\penalty0 117--134, 1992.

\bibitem{Faer}
B.~Faermann.
\newblock Localization of the {A}ronszajn-{S}lobodeckij norm and application to
  adaptive boundary element methods. {Part II}: {T}he three-dimensional case.
\newblock {\em Numer.\ Math.}, 92\penalty0 (3):\penalty0 467--499, 2002.

\bibitem{FeiKarMelPrae}
M.~Feischl, M.~Karkulik, J.~Melenk, and D.~Praetorius.
\newblock Quasi-optimal convergence rate for an adaptive boundary element
  method.
\newblock {\em SIAM J.~Numer.\ Anal.}, 51\penalty0 (2):\penalty0 1327--1348,
  2013.

\bibitem{FJ90}
M.~Frazier and B.~Jawerth.
\newblock A discrete transform and decomposition of distribution spaces.
\newblock {\em J.~Funct.\ Anal.}, 93\penalty0 (1):\penalty0 297--318, 1990.

\bibitem{G}
T.~Gantumur.
\newblock An optimal adaptive wavelet method for nonsymmetric and indefinite
  elliptic problems.
\newblock {\em J.~Comput.\ Appl.\ Math.}, 211\penalty0 (1):\penalty0 90--102,
  2008.

\bibitem{Gant13}
T.~Gantumur.
\newblock Adaptive boundary element methods with convergence rates.
\newblock {\em Numer.\ Math.}, 124:\penalty0 471--516, 2013.

\bibitem{GHS}
T.~Gantumur, H.~Harbrecht, and R.~Stevenson.
\newblock An optimal adaptive wavelet method for elliptic equations without
  coarsening.
\newblock {\em Math.\ Comput.}, 76:\penalty0 615--629, 2007.

\bibitem{GantStev}
T.~Gantumur and R.~Stevenson.
\newblock Computation of singular integral operators in wavelet coordinates.
\newblock {\em Computing}, 76:\penalty0 77--107, 2006.

\bibitem{Hac1992}
W.~Hackbusch.
\newblock {\em Elliptic {D}ifferential {E}quations: {T}heory and {N}umerical
  {T}reatment}, volume~18 of {\em Springer Series in Computational
  Mathematics}.
\newblock Springer, Berlin, 1992.

\bibitem{HackNov}
W.~Hackbusch and Z.~Nowak.
\newblock On the fast matrix multiplication in the boundary element method by
  panel clustering.
\newblock {\em Numer.\ Math.}, 54:\penalty0 463--491, 1989.

\bibitem{HaSch04}
H.~Harbrecht and R.~Schneider.
\newblock Biorthogonal wavelet bases for the boundary element method.
\newblock {\em Math.\ Nachr.}, 269--270:\penalty0 167--188, 2004.

\bibitem{HS06}
H.~Harbrecht and R.~Stevenson.
\newblock Wavelets with patchwise cancellation properties.
\newblock {\em Math.\ Comp.}, 75:\penalty0 1871--1889, 2006.

\bibitem{HU}
H.~Harbrecht and M.~Utzinger.
\newblock On adaptive wavelet boundary element methods.
\newblock Technical Report 2015-42, Mathematisches Institut, Universit{\"a}t
  Basel, Switzerland, 2015.

\bibitem{JK81}
D.~S. Jerison and C.~E. Kenig.
\newblock The {D}irichlet problem in non-smooth domains.
\newblock {\em {Ann.\ Math.~(2)}}, 113:\penalty0 367--382, 1981.

\bibitem{JK95}
D.~S. Jerison and C.~E. Kenig.
\newblock The inhomogeneous {D}irichlet problem in {L}ipschitz domains.
\newblock {\em J.~Funct.\ Anal.}, 130\penalty0 (1):\penalty0 161--219, 1995.

\bibitem{KMM07}
N.~Kalton, S.~Mayboroda, and M.~Mitrea.
\newblock Interpolation of {H}ardy-{S}obolev-{B}esov-{T}riebel-{L}izorkin
  spaces and applications to problems in partial differential equations.
\newblock In L.~{De Carli} and M.~Milman, editors, {\em Interpolation {T}heory
  and {A}pplications ({C}ontemporary {M}athematics 445)}, pages 121--177,
  Providence, RI, 2007. Amer.\ Math.\ Soc.

\bibitem{K94}
C.~E. Kenig.
\newblock {\em {H}armonic {A}nalysis {T}echniques for {S}econd {O}rder
  {E}lliptic {B}oundary {V}alue {P}roblems}, volume~83 of {\em Regional
  Conference Series in Mathematics}.
\newblock Amer.\ Math.\ Soc., Providence, RI, 1994.

\bibitem{Rokh}
V.~Rokhlin.
\newblock Rapid solution of integral equations of classical potential theory.
\newblock {\em J.~Comput.\ Phys.}, 60:\penalty0 187--207, 1985.

\bibitem{RS96}
T.~Runst and W.~Sickel.
\newblock {\em {S}obolev {S}paces of {F}ractional {O}rder, {N}emytskij
  {O}perators and {N}onlinear {P}artial {D}ifferential {E}quations}.
\newblock de Gruyter, Berlin, 1996.

\bibitem{SS11}
S.~A. Sauter and C.~Schwab.
\newblock {\em Boundary {E}lement {M}ethods}, volume~39 of {\em Springer Series
  in Computational Mathematics}.
\newblock Springer, Berlin, 2011.

\bibitem{Sch06}
R.~Schneider.
\newblock {\em Multiskalen-- und {W}avelet--{M}atrixkompression:
  {A}nalysisbasierte {M}ethoden zur effizienten {L}\"osung gro{\ss}er
  vollbesetzter {G}leichungssysteme}.
\newblock Teubner, 1998.

\bibitem{T06}
H.~Triebel.
\newblock {\em Theory of {F}unction {S}paces III}.
\newblock Birkh\"auser, Basel, 2006.

\bibitem{Manu}
M.~Utzinger.
\newblock {\em An Adaptive Wavelet Method for the Solution of Boundary Integral
  Equations in Three Dimensions}.
\newblock PhD thesis, Universi{t\"at} Basel, 2016.
\newblock In preparation.

\bibitem{V84}
G.~Verchota.
\newblock {L}ayer potentials and regularity for the {D}irichlet problem for
  {L}aplace's equation in {L}ipschitz domains.
\newblock {\em J.~Funct.\ Anal.}, 59:\penalty0 572--611, 1984.

\bibitem{Wei2016}
M.~Weimar.
\newblock Almost diagonal matrices and {B}esov-type spaces based on wavelet
  expansions.
\newblock {\em J.~Fourier Anal.\ Appl.}, 22\penalty0 (2):\penalty0 251--284,
  2016.

\end{thebibliography}

\hfill
\section*{}
\noindent Stephan \textbf{Dahlke}\\
Philipps-University Marburg \\
Faculty of Mathematics and Computer Science\\ 
Workgroup Numerics and Optimization \\
Hans-Meerwein-Stra{\ss}e, Lahnberge \\
35032 Marburg, Germany \\
dahlke@mathematik.uni-marburg.de \\
\vspace{3mm}

\noindent Helmut \textbf{Harbrecht} and Manuela \textbf{Utzinger}\\
University of Basel  \\
Department of Mathematics and Computer Science\\
Research Group of Computational Mathematics \\
Spiegelgasse 1 \\
4051 Basel, Switzerland\\
\{helmut.harbrecht, manuela.utzinger\}@unibas.ch \\
\vspace{3mm}

\noindent Markus \textbf{Weimar}\\
University of Siegen\\
Faculty IV: Science and Technology\\
Department of Mathematics\\
Walter-Flex-Stra{\ss}e 3\\
57068 Siegen, Germany\\
weimar@mathematik.uni-siegen.de
\end{document}